\documentclass{amsart}

\usepackage{enumerate}
\usepackage{amssymb}
\usepackage{amsmath,amscd}
\usepackage[colorinlistoftodos, disable]{todonotes}
\usepackage{amsthm}
\usepackage{tikz-cd}

\newcommand{\Cone}{\operatorname{Cone}}

\newcommand{\ind}{\operatorname{ind}}

\newcommand{\dist}{\operatorname{dist}}
\newcommand{\p}{\sigma}
\newcommand{\ai}{{\bar{a}}}
\newcommand{\sms}{spherical manifold submetry {}}
\newcommand{\smss}{spherical manifold submetries {}}

\newcommand{\F}{\mathcal{F}}
\newcommand{\B}{\mathcal{B}}
\newcommand{\SA}{\mathcal{S}}
\newcommand{\Hor}{\mathcal{H}}

\newcommand{\reg}{{\operatorname{reg}}}
\newcommand{\sing}{\operatorname{sing}}

\newcommand{\scal}[1]{\langle #1 \rangle}
\newcommand{\In}{\subset}

\newcommand{\Z}{\mathbb{Z}}

\newcommand{\R}{\mathbb{R}}
\newcommand{\RR}{\mathbb{R}}
\newcommand{\C}{\mathbb{C}}

\newcommand{\sphere}{\mathbb{S}}
\newcommand{\OO}{\operatorname{O}}
\newcommand{\SO}{\operatorname{SO}}

\newcommand{\Sym}{\operatorname{Sym}}
\newcommand{\Subm}{\operatorname{{\bf Subm}}}
\newcommand{\MaxLapAlg}{\operatorname{{\bf MaxLapAlg}}}

\newtheorem{theorem}{Theorem}
\newtheorem{corollary}[theorem]{Corollary}
\newtheorem{lemma}[theorem]{Lemma}
\newtheorem{proposition}[theorem]{Proposition}
\newtheorem*{conjecture}{Conjecture}

\newtheorem{maintheorem}{Theorem}

\theoremstyle{definition}
\newtheorem{definition}[theorem]{Definition}

\theoremstyle{remark}
\newtheorem{remark}[theorem]{Remark}
\newtheorem{example}[theorem]{Example}

\title[Laplacian algebras]{Laplacian algebras, manifold submetries and the Inverse Invariant Theory Problem}
\date{}

\author[R.~Mendes]{Ricardo A. E. Mendes}
\address{University of Oklahoma, USA}
\email{ricardo.mendes@ou.edu}

\author[M.~Radeschi]{Marco Radeschi}
\address{University of Notre Dame, USA}
\email{mradesch@nd.edu}

\thanks{The first-named author received support from DFG ME 4801/1-1, and DFG SFB TRR 191. The second-named author has been supported by the grant NSF 1810913. }



\begin{document}

\begin{abstract}
Manifold submetries of the round sphere are a class of partitions of the round sphere that generalizes both singular Riemannian foliations, and the orbit decompositions by the orthogonal representations of compact groups. We exhibit a one-to-one correspondence between such manifold submetries and maximal Laplacian algebras, thus solving the Inverse Invariant Theory problem for this class of partitions. Moreover, a solution to the analogous problem is provided for two smaller classes, namely orthogonal representations of finite groups, and transnormal systems with closed leaves.
\end{abstract}

\maketitle


\section{Introduction}

A \emph{manifold submetry} is a map $\p:M\to X$ from a Riemannian manifold $M$ to a metric space $X$, such that metric balls are mapped to metric balls with the same radius, and such that the preimage of every point of $X$ is a smooth, possibly disconnected, 
submanifold of $M$. Typical examples of submetries arise from taking the quotient $M\to M/G$ under the isometric action of a compact group, or the leaf space quotient $M\to M/\F$ of a singular Riemannian foliation $(M,\F)$ or, more generally, of a transnormal system.

Much like the isometric action case, the local structure of a manifold submetry $\p:M\to X$ around a point $p\in M$ is given by a manifold submetry $\p_p:V\to \Cone(Y)$ from a (real) Euclidean vector space $V$ (the \emph{slice} at $p$) 
to a metric cone $\Cone(Y)$. This is equivalent to a manifold submetry $\p_p:\sphere(V)\to Y$ from the unit sphere of $V$ to the link $Y$ of the cone. Given the central role played by these manifold submetries, we give them a special name: \emph{spherical manifold submetries}.

Given a \sms $\p:\sphere(V)\to X$, we define the subalgebra of \emph{$\p$-basic polynomials}, as the algebra generated by homogeneous polynomials over $V$ which are constant along the fibers of $\p$. In the homogeneous case 
$\p:\sphere(V)\to \sphere(V)/G$, the $\p$-basic polynomials coincide with the ring of $G$-invariant polynomials $\RR[V]^G$, which is the central object studied in Classical Invariant Theory. Recently, the authors have made progress extending results of Classical Invariant Theory to the context of singular Riemannian
foliations \cite{LR, MR, MendesRadeschi16}. 

Going in the other direction, the \emph{Inverse Invariant Theory} problem is to characterize those subalgebras 
$A\subset k[V]$ of the polynomial algebra $k[V]$ over a field $k$ that can be realized as $A=k[V]^\Gamma$ for some representation $\rho:\Gamma\to \textrm{GL}(V,k)$. This problem has been solved in positive characteristic $p$, when $k$ is a Galois field of order $p^s$ \cite[Section 8.4]{NeusSmith}(hence necessarily $\Gamma$ is finite), 
but to the best of our knowledge no such result is known in characteristic zero. The main result of this paper shows that enlarging the category to include manifold submetries allows for a satisfying answer to this problem, and hence is very natural from the Invariant Theory point of view.
\begin{maintheorem}
\label{MT:1-1correspondence}
Let $V$ be a finite-dimensional Euclidean space. Then taking the algebra of basic polynomials induces a one-one correspondence between \smss $\p:\sphere(V)\to X$ and maximal, Laplacian algebras $A\subset \R[V]$.
\end{maintheorem}

The concepts 
of maximal and Laplacian algebras are defined and discussed in section \ref{SS:Lapl}, where in particular a more precise statement of Theorem \ref{MT:1-1correspondence} is given in Theorem \ref{T:1-1correspondence}, in terms of equivalence of categories.  While the concept of maximal algebra is a bit technical, a Laplacian algebra is easily defined as a polynomial algebra $A\subset \RR[V]=\RR[x_1,\ldots x_n]$ such that $r^2=\sum_i x_i^2\in A$ and such that, for every polynomial $P\in A$, its Laplacian $\Delta P=\sum_i {\partial^2\over \partial x_i^2}P$ is in $A$ as well. Similar conditions have appeared in the literature before (cf. Remark \ref{R:Segal-Shale-Weil}), but to the best of our knowledge the concept of Laplacian algebra is new.

In view 
of the correspondence in Theorem \ref{MT:1-1correspondence}, it is natural to ask for algebraic characterizations of special classes of manifold submetries. As a first such example we consider manifold submetries with finite fibers, and provide an answer to the Inverse Invariant Theory problem for $k=\RR$ and $\Gamma$ finite:
\begin{maintheorem}
\label{MT:finite}
A subalgebra 
$A\subset \R[V]$ is of the form $A=\RR[V]^\Gamma$ for a finite group $\Gamma\subset \OO(V)$ if and only if $A$ is  maximal, Laplacian, and its  field of fractions has transcendence degree (over $\R$) equal to $\dim(V)$.
\end{maintheorem}

A manifold submetry has connected fibers if and only if the corresponding fiber decomposition is a transnormal system. Under this identification, it is possible to algebraically characterize transnormal systems with closed leaves in spheres as well:
\begin{maintheorem}
\label{MT:connected-vs-integral}
A manifold submetry $\p:\sphere(V)\to X$ has connected fibers if and only if the corresponding maximal Laplacian algebra $A$ is integrally closed in $\RR[V]$.
\end{maintheorem}

When the fibers of $\sphere(V)\to X$ are not connected, it turns out that it is possible to decompose the problem into the case of connected fibers, and the case of finite group actions.
\begin{maintheorem}
\label{MT:disconnected}
Every manifold submetry $\p:\sphere(V)\to X$ can be factored through $\sphere(V)\stackrel{\p_c}{\to} X_c\to X$, where:
\begin{enumerate}
\item $X_c$ is a metric space acted on isometrically by a finite group $G$.
\item $\sphere(V)\to X_c$ is a manifold submetry with connected fibers (equivalently, a transnormal system).
\item $X$ is isometric to $X_c/G$, and $X_c\to X$ is equivalent to the quotient map $X_c\to X_c/G$.
\end{enumerate}
\end{maintheorem}

The result above can also be interpreted as evidence that the submetry $X_c\to X$ is in some sense a Galois covering. Some previous result about (branched) coverings of Alexandrov spaces, albeit from a different point of view, can be found in \cite{HarveySearle}.

The 
authors do not know of any example of a Laplacian algebra that is not also maximal, and make the following: 

\begin{conjecture}
Every Laplacian algebra is maximal.
\end{conjecture}
As evidence we point out that this claim holds in two special but important situations, namely when the Laplacian algebra is either generated by quadratic polynomials, or by two polynomials. The former case  is essentially the main result in \cite{MendesRadeschi16}, while the latter follows from M\"unzner's results about isoparametric hypersurfaces of spheres \cite{Muenzner80}, see Section \ref{S:maximal-and-laplacian}. If this conjecture is true in general, it would have an interesting consequence in Invariant Theory: being Laplacian would be a necessary and sufficient condition for a \emph{separating} algebra of invariants to be the whole algebra of invariants. This in turn would have exciting applications for example to the study of polarizations for representations of finite groups.

{\bf The proofs}. 
The 
first part of the proof of Theorem \ref{MT:1-1correspondence} consists of showing that spherical manifold submetries are \emph{determined} by their algebras of basic polynomials, in the sense that such polynomials separate fibers, so that in particular spherical manifold submetries are objects of an \emph{algebraic nature}. This follows along the same lines as in the special case of singular Riemannian foliations, previously established in \cite{LR}, namely through the study of the \emph{averaging operator} via transverse Jacobi fields and a bootstrapping argument with elliptic regularity. 

The second part of the proof of Theorem \ref{MT:1-1correspondence} 
 is more involved. The fundamental result behind it is a procedure (Theorem \ref{T:sms construction}) that allows to build a \sms $\hat\p_A:\sphere(V)\to \hat X$ out of a Laplacian algebra $A$, without the maximality assumption. When maximality is added, this procedure is the inverse of taking basic polynomials. 
The 
\sms $\hat\p_A$ is first constructed on the regular part, and then extended to the whole sphere by metric completion. Smoothness of $\hat\p_A$ is proved using a combination of differential geometric arguments involving transverse Jacobi fields, and metric results about submetries from \cite{Lyt}. The second part of Theorem \ref{T:sms construction}, under the additional assumption that $A$ is maximal, \todo[color=cyan]{removed mention to Luna's Lemma}
relies on the fact that Laplacian algebras behave very much like algebras of invariant polynomials. More precisely, 
they admit a \emph{Reynolds operator}, which is an abstraction of the averaging operator, see Theorem \ref{T:Reynolds}.

The key 
to the proof of Theorem \ref{MT:disconnected} is producing the finite group $\Gamma$. This is done by restricting the map $X_c\to X$ to certain open dense subsets which are isometric to \emph{Riemannian orbifolds},  proving that this new map is a Galois orbifold covering, and taking $\Gamma$ to be the group of deck transformations. Theorems \ref{MT:finite} and \ref{MT:connected-vs-integral} essentially follow from Theorem \ref{MT:disconnected}.

{\bf The paper is structured as follows:} 
In Section \ref{S:correspondence} we define and discuss the categories of spherical manifold submetries and maximal Laplacian algebras, and the two functors that will establish the equivalence of the two categories, allowing us to give a formal statement of Theorem \ref{MT:1-1correspondence}.

The remainder of the paper is divided into three parts, with the first two devoted to the proof of 
Theorem \ref{MT:1-1correspondence}. Part 1 contains Sections \ref{S:qgeods} and \ref{S:spherical-ms}, and is focused on showing that the algebra of basic polynomials of a \sms is a maximal Laplacian algebra. The other direction of showing that maximal Laplacian algebras give rise to \smss is the focus of Part 2, consisting of Sections \ref{SS:duality} through \ref{S:algebras-to-smss}.

Part 3 contains Section \ref{S:disc-fibers}, about characterizing \smss with disconnected fibers, and Section \ref{S:maximal-and-laplacian}, where we  provide evidence to the Conjecture that every Laplacian algebra is maximal.

In the two final Appendices we collect facts that are either well known or that follow easily from known results: in the first appendix we collect results about isotropic and Lagrangian families of Jacobi fields along a geodesic. In the second, we lay out the basic properties of manifold submetries that closely follow those of singular Riemannian foliations.

\section*{Acknowledgements}
It is a pleasure to thank C. Lange for discussions regarding orbifold coverings, and A. Lytchak for pointing out the concept of positive reach used in Section \ref{S:algebras-to-smss}. We also thank the anonymous referee for several  suggestions that considerably improved the paper, including simplifications of the proofs of Proposition \ref{P:sing-embedded}, of Theorem \ref{T:sms construction}, and of Theorem \ref{MT:disconnected}.


\section{The correspondence between maps and algebras}\label{S:correspondence}

The goal of this section is to state a more formal version of Theorems \ref{MT:1-1correspondence} and \ref{MT:connected-vs-integral}, introducing the two categories we will work with, and defining two functors between them.

\subsection{Manifold submetries}\label{SS:manifold-submetries}

Recall that a \emph{submetry} is a continuous map $\p:X\to Y$ between metric spaces, such that for every $p\in X$ and every closed metric ball $\bar{B}_r(p)$, one has $\p(\bar{B}_r(p))=\bar{B}_r(\p(p))$.
\begin{definition}
A $C^k$-\emph{manifold submetry} is a submetry $\p:M\to X$ from a Riemannian manifold $M$ to a metric space $X$, whose fibers are $C^k$-submanifolds of $M$.
\end{definition}

Unless otherwise specified, we will work with $C^{\infty}$-manifold submetries.
This definition is slightly stronger than the definition of splitting submetry defined in \cite{Lyt}. Recall that two submanifolds $N_1, N_2$ of a Riemannian manifold $M$ are called \emph{equidistant} if for any $p, q\in N_1$, $d(p, N_2)=d(q,N_2)$ and vice versa. It is easy to check that a map $\p: M\to X$ is a manifold submetry, if and only if the fibers are smooth and equidistant. Moreover, the distance function on $X$ satisfies the following:
\begin{equation}
d_X(p_*,q_*)=d_M(\p^{-1}(p_*), \p^{-1}(q_*)).
\end{equation}

\begin{remark} It follows from the definition of manifold submetry, that:
\begin{enumerate}
\item Different fibers are allowed to have different dimension.
\item Fibers are allowed to be disconnected, but in that case the connected components of each fiber must have the same dimension.
\end{enumerate}
\end{remark}

The notion of manifold submetry is strongly related to the notions of {transnormal system} \cite{Bolton} and {singular Riemannian foliation} \cite{Molino}:
\begin{definition}\label{D:srf}
A \emph{transnormal system} is a partition $\mathcal{F}$ of a Riemannian manifold $M$ into complete, connected, injectively immersed submanifolds (called \emph{leaves}), such that every geodesic starting perpendicular to a leaf stays perpendicular to all leaves. A \emph{singular Riemannian foliation} is a transnormal system, which admits a family of smooth vector fields spanning the leaves at all points.
\end{definition}
We point out the main differences, and similarities, between these concepts:
\begin{itemize}
\item The leaves of transnormal systems can be non-closed, while the fibers of \smss can be disconnected. On the other hand, for every manifold submetry with connected fibers, the fibers define a transnormal system with closed leaves, and vice versa.
\item Given a singular Riemannian foliation $(M, \F)$, then taking closures of the leaves of $\F$ induces a transnormal system $(M,\overline{\F})$ (even better, a singular Riemannian foliation, cf. \cite{AlexandrinoRadeschiMolinoConj}). In particular, the projection $M\to M/\overline{\F}$ is a manifold submetry.
\todo{(1): Added this point}
\item Singular Riemannian foliations are, in principle, more restricted than transnormal systems. However, it is an open question whether or not every transnormal system is in fact a singular Riemannian foliation.
\end{itemize}

Given a Riemannian manifold $M$, we let $\Subm(M)$ be the category whose objects are manifold submetries $\sigma: M\to X$ and whose morphisms $(\sigma_1:M\to X_1)\to (\sigma_2:M\to X_2)$ are maps $f:X_1\to X_2$ such that $f\circ \sigma_1=\sigma_2$. We denote by $\sim$ the categorical isomorphism relation. Notice that $\p_1\sim \p_2$ if and only the partition of $M$ into $\p_1$-fibers is the same as the partition into $\p_2$-fibers.

\subsection{Maximal Laplacian algebras}\label{SS:Lapl}

Fix $V$ an $n$-dimensional Euclidean vector space, and let $\sphere(V)$ its unit sphere. Define $\RR[V]$ as the space of polynomials over $x_1\ldots x_n$ for some orthonormal basis $\{x_1,\ldots x_n\}$ of $V^*$. We define:

\begin{definition}
A polynomial algebra $A\subseteq \RR[V]$ is called \emph{Laplacian} if $r^2:=\sum_i x_i^2$ belongs to $A$, and for every $f\in A$, the Laplacian $\Delta f=\sum_i{\partial ^2\over \partial x_i^2}(f)$ belongs to $A$ as well.
\end{definition}
Notice that $r^2$ and $\Delta$ do not depend on the specific choice of orthonormal basis $x_1,\ldots x_n$, and thus are well defined.

\begin{remark}\label{R:Segal-Shale-Weil}
The operators $\Delta, r^2: \RR[V]\to \RR[V]$ induce a well-known action of $\mathfrak{sl}(2,\RR)$ on $\RR[V]$ related to the \emph{Segal-Shale-Weil representation}, cf. \cite[Ch. 2.1]{HT}. From this point of view, Laplacian algebras are simply $\mathfrak{sl}(2,\RR)$-invariant subalgebras of $\RR[V]$.
\end{remark}

\begin{definition}
\label{D:maximal}
Given a subalgebra
$A\subset\RR[V]$, define the equivalence relation $\sim_A$ on $\sphere(V)$ by letting $p\sim_A q$ if $f(p)=f(q)$ for every $f\in A$.
The algebra $A$ is called \emph{maximal} if it cannot be enlarged without changing the relation $\sim_A$. In other words, for every $f\notin A$ there exist $p,q\in \sphere(V)$ such that $p\sim_A q$ but $f(p)\neq f(q)$.
\end{definition}

We define $\MaxLapAlg(V)$ the category whose objects are the maximal Laplacian subalgebras of $\RR[V]$, and the morphisms are simply the inclusions $A_1\subseteq A_2$. 

\begin{remark}
If a group $G$ acts orthogonally on $V$, it leaves the Laplace operator fixed, so that the algebra $A=\RR[V]^G$ of invariant polynomials is a Laplacian algebra. If $G$ is additionally compact, then $A$ separates $G$-orbits, and hence $x\sim_A y$ if and only if $x,y$ belong to the same $G$-orbit. It follows immediately from Definition \ref{D:maximal} that $A$ is maximal.
\end{remark}

\subsection{Correspondence}
Given a map $\p:\sphere(V)\to X$ onto some set $X$, we define $\mathcal{B}(\p)\subset \RR[V]$ the algebra of  \emph{$\p$-basic polynomials}, that is, the algebra generated by homogeneous polynomials which are constant on the fibers of $\p$. On the other hand, given $A\subset\RR[V]$, define the set $X_A=\sphere(V)/\sim_A$ (with $\sim_A$ as in Definition \ref{D:maximal}), and $\mathcal{L}(A)=\p_A:\sphere(V)\to X_A$ the natural quotient map.
If the fibers of $\p_A$ are equidistant, then $X_A$ can be given the structure of a metric space, by defining $d_{X_A}(p_*,q_*)=d_{\sphere(V)}(\p_A^{-1}(p_*), \p_A^{-1}(q_*))$. With respect to this metric structure, $\p_A$ becomes a submetry.

We can finally restate Theorems \ref{MT:1-1correspondence} and \ref{MT:connected-vs-integral}:
\begin{theorem}[Theorem \ref{MT:1-1correspondence}]\label{T:1-1correspondence}
For any Euclidean vector space $V$, the maps $\mathcal{B}$, $\mathcal{L}$ above define contravariant functors
\[
\begin{tikzcd}[ampersand replacement=\&]
\Subm(\sphere(V))/\sim \arrow[bend right]{r}{\mathcal{B}} \& 
\MaxLapAlg(V)\arrow[bend right]{l}{\mathcal{L}}
\end{tikzcd}
\]
which provide an equivalence between the two categories.
\end{theorem}

Recalling that a manifold submetry  with connected fibers defines a partition into fibers which is a transnormal system, we can restate Theorem \ref{MT:connected-vs-integral} as follows

\begin{theorem}[Theorem \ref{MT:connected-vs-integral}]
For any Euclidean vector space $V$, the maps $\mathcal{B}$, $\mathcal{L}$ above define a bijection
\[
\begin{tikzcd}[ampersand replacement=\&]
\left\{\begin{array}{c}\textrm{transnormal systems}\\ \textrm{with closed leaves}\\ \textrm{in }\sphere(V)\end{array} \right\} \arrow[bend right]{r}{\mathcal{B}} \& 
\left\{\begin{array}{c}\textrm{Maximal Laplacian}\\\textrm{algebras }A\subseteq \RR[V]\\ \textrm{integrally closed in }\RR[V]\end{array} \right\}\arrow[bend right]{l}{\mathcal{L}}
\end{tikzcd}
\]
\end{theorem}

\part{From manifold submetries to Laplacian algebras}

The first part, comprising Sections \ref{S:qgeods} and \ref{S:spherical-ms}, aims to show that given a manifold submetry $\p:\sphere(V)\to X$, the algebra of basic polynomials is a maximal Laplacian algebra. Combining recent results in the theory of singular Riemannian foliations (cf. \cite{AR, LR}), one can prove this result in the case where $\p$ is the quotient map $\sphere(V)\to \sphere(V)/\F$ for some singular Riemannian foliation $(\sphere(V), \F)$. Here the strategy is to extend such results to the more general case of manifold submetries. Some results, such as the Homothetic Transformation Lemma, extend with minimal changes from the original version for singular Riemannian foliations: all such results have been added in Appendix \ref{A:manif-subm}. Other results, such as equifocality, require an original approach, and these form the bulk of the next two sections.

\section{Spherical Alexandrov spaces, quotient geodesics, and submetries}\label{S:qgeods}

\subsection{Alexandrov spaces}

Alexandrov spaces are a certain class of metric spaces $(X,d)$ with a lower curvature bound. We will assume that the reader is familiar with this concept, and we refer to \cite{BGP} for an introduction to the subject.

Given an Alexandrov space $(X,d)$, a point $x\in X$, and a sequence of positive real numbers $r_i$ converging to zero, the sequence of rescaled pointed metric spaces $(X,x, r_i\cdot d)$ converges in the Gromov-Hausdorff sense to a pointed metric space of non-negative curvature $(T_xX,o)$ called the \emph{tangent space to $X$ at $x$} and its elements are called \emph{tangent vectors}, even though this space is in general not a vector space. For a vector $v\in T_xX$, one defines the norm $|v|=d(o,v)$. The subset of $T_xX$ of vectors with norm 1 is again an Alexandrov space called the \emph{space of directions} $\Sigma_xX$, and $T_xX$ is in fact the metric cone over $\Sigma_xX$.

Recall that a geodesic in a metric space $X$ is a curve $\gamma:[a,b]\to X$ parametrized by arc length, minimizing the distance between the end points. If $X$ is an Alexandrov space, then for every $t_0\in [a,b]$
\todo{(2): Corrected $t_0$ typo}
one defines the forward velocity $\gamma^+(t_0)\in T_{\gamma(t_0)}X$. Almost every unit-norm vector is the velocity of a geodesic. Furthermore, if two geodesics have the same initial velocity, then they coincide for as long as they are both defined.

\subsection{Infinitesimal submetry}\label{SS:inf-subm}

Let $\p:M\to X$ be a manifold submetry, and $p\in M$. From \cite{Lyt}, the sequence of rescalings $\p_r: (M, rg, p)\to (X, rd_X, \p(p))$, as $r\to 0$, converges to a submetry $d_p\sigma: T_pM\to T_{\p(p)}X$, called the \emph{differential of $\p$ at $p$}, such that for every $r\in \RR_+$, $d_p\p(r\cdot v)=r\cdot  d_p\p(v)$. Moreover, letting $V_p=T_pL_p$ ($L_p$ the $\p$-fiber through $p$) and $\Hor_p=V_p^\perp$, the restriction $\p_p:=d_p\p|_{\Hor_p}:\Hor_p\to T_{\p(p)}X$ is again a submetry, whose fibers are (the blow up of) the intersections between the fibers of $\p$ and the slice $D_p:=\exp \nu^{<\epsilon}_pP$. By Lemma \ref{L:transverse-int} in the Appendix, these are manifolds, and thus $\p_p$ is a manifold submetry.

By construction, the preimage of the vertex in  ${T_{\p(p)}X}$ is simply the vertex in $\Hor_p$, and since $\p_p$ is a submetry it follows that all the $\p_p$-fibers are contained in the distance spheres of $\Hor_p$ around the origin. Denoting $\sphere_p$ the unit sphere in $\Hor_p$ and $\Sigma_{\p(p)}X$ the space of directions of $X$ at $\p(p)$, the map $\p_p$ then restricts to a manifold submetry $\sphere_p\to \Sigma_{\p(p)}X$, which we call the \emph{infinitesimal submetry of $\p$ at $p$} and still denote by $\p_p$.

\subsection{Horizontal geodesics}


Given a manifold submetry $\p:M\to X$ and $p\in M$, a tangent vector $v\in T_pM$ is called \emph{horizontal} if it is perpendicular to $T_p\p^{-1}(\p(p))$. A geodesic $\gamma:[a,b]\to M$ is a \emph{horizontal geodesic}, if $\gamma'(t)$ is horizontal for every $t\in [a,b]$. It is known (cf. \cite{Lyt}, Lemma 5.4) that for every vector $w\in \Sigma_{x}X$, every point $p\in \p^{-1}(x)$, and every horizontal vector $v\in d_p\sigma^{-1}(w)$, the geodesic  $\gamma(t):=\exp_p(tv)$ is a horizontal geodesic. Furthermore, the projection of a horizontal geodesic is concatenation of geodesics on $X$.

\subsection{Spherical Alexandrov spaces}
One fundamental property of singular Riemannian foliation is the so-called equifocality, which states that if two horizontal geodesics $\gamma_1, \gamma_2:(a,b)\to M$ are so that $\gamma_1(t)$ and $\gamma_2(t)$ belong to the same leaf for every $t$ in some open set $(a',b')\subseteq (a,b)$, then in fact $\gamma_1(t)$ and $\gamma_2(t)$ belong to the same leaf for every $t\in (a,b)$. This property was proved for singular Riemannian foliations in \cite{LT} and \cite{AT}, in both cases using the existence of smooth vector fields spanning the leaves.

In this section we prove equifocality for manifold submetries, and to do so we prove that the Alexandrov spaces which occur as bases of manifold submetries have very special properties which allow to define geodesics even after they stop minimizing.

We begin by defining some special classes of Alexandrov spaces. These definitions are by induction on the dimension. Let $\B_1$ be the class of closed 1-dimensional Alexandrov spaces, namely circles $S^1$, closed intervals $[a,b]$, the real line $\RR$ and the half line $[0,\infty)$.
Given $X\in \B_1$, a \emph{quotient geodesic} on $X$ is a $1$-Lipschitz map $\gamma:[0,\ell]\to X$, with a partition $0\leq t_1<t_2<\ldots <t_N\leq \ell$, such that
\begin{enumerate}
\item Each restriction $\gamma|_{[t_i,t_{i+1}]}$ is a locally minimizing geodesic.
\item For every $i=1,\ldots N$, $\gamma(t_i)$ is in the boundary of $X$.
\end{enumerate}
In other words, quotient geodesics are ``geodesics which bounce back and forth''.

Finally, $X\in \B_1$ is called a \emph{spherical Alexandrov space} if it admits an involutive isometry $a:X\to X$ such that, for every $x\in X$ and $v\in \Sigma_xX$, the quotient geodesic $\gamma(t)$ with $\gamma(0)=x$, $\gamma'(0)=v$ satisfies $\gamma(\pi)=a(x)$. Let $\SA_1$ be the set of spherical Alexandrov spaces. It is easy to see that $\SA_1$ consists of intervals $[0,\pi/k]$ and circles $S^1$ of length $2\pi/k$, for $k$ positive integer.

Assume the classes $\B_j$, $\SA_j$ of $j$-dimensional Alexandrov spaces have been defined, for $j=1,\ldots m-1$.
\begin{definition}
Let $X$ be an $m$-dimensional Alexandrov space. Then:
\begin{itemize}
\item We say that $X$ is \emph{base-like} if every tangent vector exponentiates to a geodesic, and for every $x\in X$, $\Sigma_xX\in \SA_{m-1}$.
\item We denote $\B_m$ the set of base-like Alexandrov spaces of dimension $m$.
\item Given $X\in \B_m$, then fixing  $x\in X$ and $v\in \Sigma_xX$, we define a \emph{quotient geodesic} from $(x,v)$ as a 1-Lipschitz map $\gamma:[0,\ell]\to X$ with a partition $0\leq t_1<t_2<\ldots <t_N\leq \ell$, such that
\begin{enumerate}
\item Each restriction $\gamma|_{[t_i,t_{i+1}]}$ is a locally minimizing geodesic.
\item For every $i=1,\ldots N$, $\gamma^+(t_i)=a(\gamma^-(t_i))$, where $\gamma^{\pm}(t_i)\in \Sigma_{\gamma(t_i)}X$ are the left and right limit of $\gamma$ at $t_i$ and $a: \Sigma_{\gamma(t_i)}X\to  \Sigma_{\gamma(t_i)}X$ is the involutive isometry which exists since $\Sigma_{\gamma(t_i)}X\in \SA_{m-1}$.
\end{enumerate}
\item Given $X\in \B_m$, we say that $X$ is a \emph{spherical Alexandrov space} if it admits an involutive isometry $a:X\to X$ (called \emph{antipodal map}) such that, for every $x\in X$ and $v\in \Sigma_xX$, the quotient geodesic $\gamma(t)$ from $(x,v)$ satisfies $\gamma(\pi)=a(x)$ (independent of $v$).
\item Define $\SA_m$ the set of $m$-dimensional, spherical Alexandrov spaces.
\end{itemize}
\end{definition}
\todo{(8): added the remark}
\begin{remark}
By induction, it is easy to see that the antipodal map $a$ for a spherical Alexandrov space is unique.
\end{remark}
Then we have the following:

\begin{lemma}[Uniqueness of quotient geodesics]\label{L:Qgeod-uniq.}
Given a base-like Alexandrov space $X$, and two quotient geodesics $\gamma_i:[0, \ell_i]\to X$, $i=1,2$ with $\gamma_1(0)=\gamma_2(0)$ and $\gamma_1^+(0)=\gamma_2^+(0)$, then $\gamma_1(t)=\gamma_2(t)$ for any $t\in [0,\min\{\ell_1,\ell_2\}]$.
\end{lemma}
\begin{proof}
Let $\ell=\min\{\ell_1,\ell_2\}$. It is enough to prove that the set $J=\{t\in [0,\ell]\mid \gamma_1(t)=\gamma_2(t)\}$ is open and closed. Since $J$ is clearly closed, it is enough to show that it is open. Suppose then that $[0,t_0]\subset J$. If $t_0>0$, then $\gamma_1^-(t_0)=\gamma_2^-(t_0)$, therefore $\gamma_1^+(t_0)=a(\gamma_1^-(t_0))=a(\gamma_2^-(t_0))=\gamma_2^+(t_0)$. If $t_0=0$, then $\gamma_1^+(t_0)=\gamma_2^+(t_0)$ by assumption.

In either case, there is a $\delta>0$ small enough that $\gamma_1|_{[t_0,t_0+\delta)}$ and $\gamma_2|_{[t_0,t_0+\delta)}$ are geodesics with the same initial direction, and therefore they are equal. Thus $[0,t_0+\delta)\subset J$ and $J$ is open.
\end{proof}

\todo{(9): corrected typos}
\begin{remark}
The notion of quotient geodesics, and their properties, have also been discussed and proved in \cite{LT} in the context of singular Riemannian foliations, see Definition \ref{D:srf} and the discussion below it.
\end{remark}

\begin{proposition}\label{P:facts}
Let $M$ be a complete Riemannian manifold, and $\p:M\to X$ a manifold submetry onto an $m$-dimensional Alexandrov space. Then:
\begin{enumerate}
\item $X\in \B_m$.
\item A horizontal geodesic in $M$ projects to a quotient geodesic in $X$.
\item If $M=\sphere^n$ is the unit sphere of curvature 1, then $X\in \SA_m$ and the antipodal map is $a(\p(v))=\p(-v)$.
\item For any two points $p_1, p_2\in M$ with $\p(p_1)=\p(p_2)=p_*$, and vectors $v_i\in \sphere_{p_i}$ with $\p_{p_1}(v_1)=\p_{p_2}(v_2)$, the geodesics $\gamma_1(t)=\exp(tv_1)$ and $\gamma_2(t)=\exp(tv_2)$ satisfy $\p(\gamma_1(t))=\p(\gamma_2(t))$ for all $t$.
\end{enumerate}
\end{proposition}
\begin{proof} We prove it by induction on the dimension $n$ of $M$. If $n=1$, then $M$ is either $\RR$ or $\sphere^1$ and the only nontrivial case is for $X$ to have dimension 1 as well. In this case, $X\in \B_1$ trivially, and it is easy to see that $\p:M\to X$ is a local isometry away from $\p^{-1}(\partial X)$, and in fact a quotient geodesic. In particular, if $M$ is the unit circle $\sphere^1$, one defines $a: X\to X$ by $a(\p(p))=\p(-p)$, and it is easy to see that $X\in \SA_1$.

Suppose now that the result holds for any manifold submetry $N\to Y$ with $\dim N\leq n-1$, and take $\p:M\to X$ with $\dim M=n$, and let $m=\dim X$. We make two observations:
\todo{(10a): changed $a,b$ into $t_1, t_2$ throughout the proof.}
\begin{enumerate}[a.]
\item Fixing a horizontal geodesic $\gamma:[0,\ell]\to M$, the projected curve $\gamma_*(t)=\p(\gamma(t))$ satisfies $\textrm{length}(\gamma_*|_{[t_1,t_2]})=\textrm{length}(\gamma|_{[t_1,t_2]})$, for any $[t_1,t_2]\subseteq [0,L]$. In particular, for any $t_1\in [0,\ell]$ there is some $t_2>t_1$ such that $\gamma|_{[t_1,t_2]}$ minimizes the distance between the fibers at $\gamma(t_1)$ and $\gamma(t_2)$, and therefore
\[
\textrm{length}(\gamma_*|_{[t_1,t_2]})=\textrm{length}(\gamma|_{[t_1,t_2]})=d_{M}(L_{\gamma(t_1)}, L_{\gamma(t_2)})=d_X(\gamma_*(t_1),\gamma_*(t_2)),
\]
which implies that $\gamma_*|_{[t_1,t_2]}$ is a geodesic.
\item For any point $p_*\in X$, and $p\in\p^{-1}(p_*)$, the differential of $\sigma$ at $p$ defines a  manifold submetry $\p_p:\sphere_p\to \Sigma_{p_*}X$ where $\sphere_p$ is the unit sphere in $\nu_p(\p^{-1}(p_*))$ (cf. \ref{SS:inf-subm}). Since $\dim \sphere_p<n$, it follows by induction that $\Sigma_{p_*}X\in \SA_{m-1}$.
\end{enumerate}
\todo{(10b): ``lemma'' changed into ``proposition''}
We proceed to prove the proposition:

1) For any $p_*\in X$, $p\in \p^{-1}(p)$ and $v_*\in \Sigma_{p_*}X$, one can find a $v\in \sphere_p$ such that $\p_p(v)=v_*$. Since the fibers of $\p$ are closed, there is a constant $\epsilon$ such that $\gamma(t)=\exp tv$ satisfies
\[
d(L_p, L_{\gamma(t)})=d(p,L_{\gamma(t)})=\textrm{length}(\gamma|_{[0,t]})\qquad \forall t\in(0,\epsilon).
\]
By point a. it follows that $v_*$ exponentiates to a geodesic $\gamma_*(t):=\p(\gamma(t))$ for $t<\epsilon$. Together with point b., it follows that $X\in \B_{m}$.

2) By observation a., any horizontal geodesic $\gamma$ in $M$ is projected to a curve $\gamma_*$ which is a piecewise geodesic. Thus it remains to prove that $\gamma_*^+(t)=a(\gamma_*^-(t))$ for every $t$. Fix a point $p=\gamma(t_0)$, and let $v=\gamma'(t_0)$. Then $\gamma_*^+(t_0)=\p_p(v)$, and $\gamma_*^-(t_0)=\p_p(-v)$, thus we need to prove that the antipodal map $a:\Sigma_pX\to \Sigma_pX$ satisfies $a(\p_p(v))=\p_p(-v)$. For any $w_*\in \Sigma_v\left(\Sigma_pX\right)$, the quotient geodesic $\psi:[0,\pi]\to \Sigma_pX$ with $\psi^+(0)=w_*$ is, by the induction step, given by $\p_p(\cos(t) v+\sin(t)w)$, where $w\in T_{v}\sphere_p\simeq \langle v \rangle^{\perp}$ is the vector projecting to $w_*$. Therefore,
\[
a(\p_p(v))=\psi(\pi)=\p_p(-v).
\]

3) Let $p_*=\p(p)\in X$, $v_*\in \Sigma_pX$ and let $v\in \sphere_p$ be such that $\p_p(v)=v_*$. By point (2) the quotient geodesic $\gamma_*$ with $\gamma^+(0)=v_*$ is $\gamma_*(t)=\p(\exp(tv))$. Since in this case $M$ is a round sphere of curvature 1, $\exp(tv)=\cos(t)p+\sin(t)v$ and $\gamma_*(t)=\p(\exp(\pi v))=\p(-p)$, independently of $v_*$.

4) By point (2), $(\gamma_1)_*(t)=\p(\gamma_1(t))$ and $(\gamma_2)_*(t)=\p(\gamma_2(t))$ are both quotient geodesics, and by hypothesis $(\gamma_1)_*^+(0)=(\gamma_2)_*^+(0)$. Since there is no 
\todo{(10c): ``splitting'' to ``branching''}
geodesic branching in Alexandrov spaces, it also follows that quotient geodesics are uniquely determined by their initial vector, and therefore $(\gamma_1)_*(t)=(\gamma_2)_*(t)$ for all $t$.
\end{proof}

\begin{proposition}\label{P:Jacfields}
Let $\p:M\to X$ be a $C^2$-manifold submetry, $\gamma:[0,\ell]\to M$ a horizontal geodesic, and $L_t$ the fiber through $\gamma(t)$. Then there is a vector space $W$ of Jacobi fields along $\gamma$, such that:
\begin{itemize}
\item $W$ is \emph{isotropic}, i.e. for any $J_1, J_2\in W$, $\langle J_1(t),J'_2(t)\rangle=\langle J_1'(t),J_2(t)\rangle$.
\item For any $t\in [0,\ell]$,
\[
T_{\gamma(t)}L_t=W(t):=\{J(t)\mid J\in W\}
\]
\end{itemize}
\end{proposition}
\begin{proof}
Let $P_0\subset L_0$ be a relatively compact open set of $L_0$ containing $\gamma(0)$, and let $\epsilon$ small enough, that the normal exponential map $\exp:\nu^{<\epsilon}P_0\to M$ is a $C^2$-diffeomorphism. Fixing $\delta<\epsilon$, any vector $w\in T_{\gamma(\delta)}L_\delta$ is the initial vector of some curve $\alpha_w:(-1,1)\to L_\delta$ with $\alpha_w'(0)=w$. We can write $\alpha_w(s)=\exp(\delta v(s))$, where $v(s)$ is a curve of unit normal vectors in $\nu P_0$. We can then define the family of horizontal geodesics $\gamma_s(t)=\exp(tv(s))$, the Jacobi field $J_w(t)={d\over ds}\big|_{s=0}\gamma_s(t)$, and define $W=\{J_w\mid w\in T_{\gamma(\delta)}L_\delta\}$. It is easy to check that $w\mapsto J_w$ is a linear map, and $W$ is a vector space.

We first prove that $W$ is isotropic. Recall that for any two Jacobi fields $J_1, J_2$, the function $\langle J_1(t),J'_2(t)\rangle-\langle J_1'(t),J_2(t)\rangle$ is constant on $t$, thus it is enough to check that it vanishes at a single time $t=\delta$. Given $J_{v_1},J_{v_2}\in W$, we have
\begin{align*}
\langle J_{v_1}'(\delta), J_{v_2}(\delta)\rangle&=\langle \nabla_{\gamma'(\delta)}J_{v_1}, J_{v_2}(\delta)\rangle=\langle S_{\gamma'(\delta)}J_{v_1}(\delta), J_{v_2}(\delta)\rangle=\langle S_{\gamma'(\delta)}v_1,v_2\rangle
\end{align*}
where $S_{\gamma'(\delta)}$ denotes the shape operator of $L_{\delta}$ in the direction of $\gamma'(\delta)$. Since the shape operator is symmetric, it follows that $\langle J_{v_1}'(\delta), J_{v_2}(\delta)\rangle-\langle J_{v_2}'(\delta), J_{v_1}(\delta)\rangle=0$, and thus $W$ is isotropic.

We now check that the equality $T_{\gamma(t)}L_t=W(t)$ holds for all $t$. Letting $\gamma_*(t)=\p(\gamma(t))$, the family of geodesics $\gamma_s(t)$ above defining $J\in W$ satisfies $\p(\gamma_s(\delta))=\gamma_*(\delta)$ for all $s$. By the Homothetic Transformation Lemma (cf. Lemma \ref{L:HTL} in the Appendix), $\p(\gamma_s(t))=\gamma_*(t)$ for all $s$ and all $t\in (0,\delta)$. In particular, $d_{\gamma_s(0)}\p(\gamma_s'(0))=\gamma_*^+(0)$, and by Proposition \ref{P:facts} it follows that $\p(\gamma_s(t))=\gamma_*(t)$ for all $s$ and for \emph{all} $t\in [0,\ell]$. In particular, $J(t)\in T_{\gamma(t)}L_t$ for any $J\in W$, and $W(t)\subseteq T_{\gamma(t)}L_t$, for all $t\in [0,\ell]$. Furthermore, $W(t)= T_{\gamma(t)}L_t$ for all $t\in (0,\delta)$ and, by part 1 of Lemma \ref{L:transverse-int}, equality holds for $t=0$ as well. We now show equality for all $t$, by showing that the set $$I=\{s\in[0,\ell]\mid W(t)=T_{\gamma(t)}L_t\quad\forall t\in [0,s]\}$$
is open and closed.
To prove it is closed, suppose $[0,t_0)\subseteq I$, and pick $\delta'$ small enough that $\dim W(t_0-\delta')=\dim W$ and such that the Homothetic Transformation Lemma can be applied in a $\delta'$-neighborhood of $L_{t_0}$. For any $w=J(t_0-\delta')\in W(t_0-\delta')$ and any $\lambda\in [0,1]$, we claim that the homothetic transformation $h_\lambda$ around $L_{t_0}$ satisfies
\[
(h_{\lambda})_*(w)=J(t_0-\lambda \delta').
\]
In fact, letting $\gamma_s(t)$ the family of horizontal geodesics such that $J(t)={d\over ds}\Big|_{s=0}\gamma_s(t)$, we know that for every $s$, $\gamma_s(t_0)$ belongs to $L_{t_0}$ and $\psi_s(t):=\gamma_s(t_0-t)$ is the minimizing segment between $\gamma_s(t_0-\delta)$ and $L_{t_0}$. In particular, $h_\lambda(\gamma_s(t_0-\delta'))=h_\lambda(\psi_s(\delta'))=\psi_s(\lambda\delta')=\gamma_s(t_0-\lambda\delta')$. The claim follows by differentiating this equation with respect to $s$. Therefore, $(h_\lambda)_*W(t_0-\delta')=W(t_0-\lambda \delta')$ and for $\lambda=0$ we have $(h_0)_*W(t_0-\delta')=W(t_0)$. On the other hand, by part 1) of Lemma \ref{L:transverse-int}, we also have $$(h_0)_*W(t_0-\delta')=(h_0)_*T_{\gamma(t_0-\delta')}L_{t_0-\delta'}=T_{\gamma(t_0)}L_{t_0},$$ and thus $t_0\in I$ as well.

To prove that $I$ is open, we use the fact that, since $W$ is isotropic, for every $t_0\in[0,\ell]$ there is a $\delta$ such that $\dim W(t)=\dim W$ for all $t\in (t_0-\delta,t_0+\delta)\setminus \{t_0\}$ (cf. Proposition \ref{P:equality-ae}). Thus if $t_0\in I$ then $\dim W=\dim L_t$ for all $t\in (t_0-\delta,t_0)$, and we need to prove that $\dim W=\dim L_t$ for every $t\in (t_0,t_0+\delta)$ as well. We prove so by contradiction: suppose $\dim W< \dim L_{t'}$ for some $t'\in (t_0,t_0+\delta)$. Then by repeating the same arguments as above around $t'$, there is an isotropic subspace $W'$ of Jacobi fields such that $W'(t)\subseteq T_{\gamma(t)}L_{t}$ for all $t$, and $W'(t')= T_{\gamma(t')}L_{t'}$. In particular, $\dim W'> \dim W$. However, for all but finitely many values of $t\in (t_0-\delta,t_0)$, one has
\[
\dim T_{\gamma(t)}L_{t}\geq \dim W'(t)>\dim W(t)=\dim T_{\gamma(t)}L_{t},
\]
giving a contradiction.
\end{proof}

As a corollary of the results in this section, we have
\begin{proposition}\label{P:regular-part-convex}
Let $\p:M\to X$ a manifold submetry, let $M^\reg\subseteq M$ denote the stratum of fibers with maximal dimension, and let $X^\reg=\p(M^\reg)$. Then $X^\reg$ is convex in $X$.
\end{proposition}
\begin{proof}
Let $p_*, q_*\in X^\reg$ and let $\gamma_*:[0,1]\to X$ a minimizing geodesic between $p_*$ and $q_*$. We need to prove that $\gamma_*(t)\in X^\reg$ for all $t\in [0,1]$.
Let $L_p=\p^{-1}(p_*)$, $L_q=\p^{-1}(q_*)$ and let $\gamma:[0,1]\to M$ be a horizontal geodesic projecting to $\gamma_*$. Clearly $\gamma$ minimizes the distance between $L_p$ and $L_q$. Suppose by contradiction that for some $t_0\in (0,1)$, $\gamma(t_0)$ is contained in a fiber of non-maximal dimension. By Proposition \ref{P:Jacfields} the tangent spaces of fibers along $\gamma$ are spanned by an isotropic subspace of Jacobi fields, and by standard results on isotropic subspaces of Jacobi fields (see Appendix \ref{A:Lagr}) the dimension of the fiber $L_t$ through $\gamma(t)$ is maximal for all but discretely many values of $t$. By Lemma \ref{L:transverse-int}, for $\epsilon$ small enough, the closest-point projection map $L_{t_0+\epsilon}\to L_{t_0}$ is a submersion. Since by assumption $\dim L_{t_0}<\dim L_{t_0+\epsilon}$, the fiber of $L_{t_0+\epsilon}\to L_{t_0}$ through $\gamma(t_0+\epsilon)$ contains at least another point, call it $\bar{p}$. Let $\bar \gamma:[t_0,1]\to M$ be the horizontal geodesic such that $\bar\gamma(t_0)=\gamma(t_0)$ and $\bar{\gamma}(t_0+\epsilon)=\bar{p}$. Then $\bar\gamma_*(t):=\p\circ\bar\gamma(t)$ equals $\gamma_*(t)$ at $t=t_0$ and $t_0+\epsilon$. By the Homothetic Transformation Lemma, $\bar\gamma_*(t)=\gamma_*(t)$ for $t\in [t_0,t_0+\epsilon]$, and thus by Proposition \ref{P:facts} (4), $\bar\gamma_*(t)=\gamma_*(t)$ for every $t\in [t_0,1]$. But then the concatenation $\gamma|_{[0,t_0]}\star\bar\gamma$ is a (non-minimizing) curve from $L_p$ to $L_q$ with the same length of the (minimizing) curve $\gamma$, contradiction.
\end{proof}

\section{Spherical manifold submetries}\label{S:spherical-ms}

A \emph{spherical manifold submetry} is a manifold submetry from a round sphere of curvature 1. Given a \sms $\sphere^n\to X$, we have from Proposition \ref{P:facts} that $X$ is a spherical Alexandrov space. The goal of this section is to prove the first part of Theorem \ref{MT:1-1correspondence}: namely, we prove that given a Euclidean vector space $V$ and a $C^2$-manifold submetry $\p:\sphere(V)\to X$ from the unit sphere of $V$, there exists a maximal Laplacian algebra $A:=\RR[V]^\p$ whose level sets are the fibers of $\p$.

\begin{proposition}[Basic mean curvature]
Let $\p:\sphere^n\to X$ be a $C^2$ \sms. Then the mean curvature vector field of $\p$ is basic. That is, for any $p_1, p_2\in \sphere^n$ in the same fiber $L$ of maximal dimension, the mean curvature vectors $H_1, H_2$ of $L$ at $p_1$, $p_2$ respectively, satisfy $d_{p_1}\p(H_1)=d_{p_2}\p(H_2)$.
\end{proposition}

\begin{proof}
It is enough to prove that, given two points $p_1$, $p_2$ with $\p(p_1)=\p(p_2)=p_*$ and vectors $v_i\in \sphere_{p_i}$, with $d_{p_i}\p(v_i)=v_*$, one has that the shape operator of $L=\p^{-1}(p_*)$ satisfies $\textrm{tr}(S_{v_1})=\textrm{tr}(S_{v_2})$. In fact, we claim that $S_{v_1}$ and $S_{v_2}$ have the same eigenvalues.
The proof of this fact, is essentially the same as \cite[Proposition 3.1]{AR}, and it hinges on the following facts:
\begin{itemize}
\item Letting $\gamma_i(t)=\exp(tv_i)$, $i=1,2$, define the spaces $\Lambda_i$ of Jacobi fields along $\gamma_i$ given by
\[
\Lambda_i=\{J(t)\mid J(0)\in T_{p_i}L,\, J'(0)=-S_{\gamma(0)}J(0)\}\oplus\{J(t)\mid J(0)=0, J'(0)\perp \gamma_i'(0)\oplus T_{p_i}L\}.
\]
These are the Lagrangian spaces of Jacobi fields (see Appendix \ref{A:Lagr}) consisting of Jacobi fields generated by variations of $\gamma_i$ via horizontal geodesics through $L$. Their focal functions $f_{\Lambda_i}(t)=\dim\{J\in \Lambda_i\mid J(t)=0\}$ have the property that $\lambda$ is an eigenvalue of $S_{v_i}$ with multiplicity $m$, if and only if $f_{\Lambda_i}(\arctan(1/\lambda))=m$.
\item The spaces $W_i$ of Jacobi fields along $\gamma_i$ defined in the Proposition \ref{P:Jacfields} are clearly contained in $\Lambda_i$. By equation \eqref{E:focal-sum} in Appendix \ref{A:Lagr}, the following formulas for the focal functions hold:
\[
f_{\Lambda_i}(t)=f_{W_i}(t)+f_{\Lambda_i/W_i}(t).
\]
\item By Proposition \ref{P:Jacfields}, the function $f_{W_i}(t)$ can be rewritten as
\[
f_{W_i}(t)=\dim W_i - \dim W_i(t)=\left(\max_{t\in \RR}\dim L_{\gamma_i(t)}\right)-\dim L_{\gamma_i(t)}.
\]
Since $\gamma_1(t)$ and $\gamma_2(t)$ are contained in the same leaves for every $t$, clearly $f_{W_1}(t)=f_{W_2}(t)$ for every $t$.
\item Using Wilking's Transverse Jacobi Equation (see Example \ref{E:example} in  Appendix \ref{A:Lagr}) the curvature operators $R^{H_i}(t)$ of the quotient bundles $H_i=E/E_{W_i}$ can be identified, for all but discretely many $t\in I$, with the Riemann curvature operator of $X$ along $\gamma_*(t)$. By continuity, $R^{H_1}(t)=R^{H_2}(t)$ and, in particular, $f_{\Lambda_1/W_1}(t)=f_{\Lambda_2/W_2}(t)$.
\end{itemize}
Summing up, we have $f_{\Lambda_1}(t)=f_{\Lambda_2}(t)$ for all $t$, and therefore the eigenvalues of $S_{v_1}$, $S_{v_2}$ agree.
\end{proof}

\begin{proposition}\label{P:homog-pol}
Let $\p:\sphere^n\to X$ be a $C^2$ \sms and let $A\subset \RR[x_0,\ldots x_n]$ be the algebra generated by the homogeneous polynomials \todo[color=cyan]{added ``generated by''} which are constant along the fibers of $\p$. Then
\begin{itemize}
\item $A$ is finitely generated.
\item Letting $\rho_1,\ldots, \rho_k$ generators of $A$ and $\rho=(\rho_1,\ldots, \rho_k):\sphere^n\to \RR^k$, then the fibers of $\p$ coincide with the fibers of $\rho$.
\item Letting $X'=\p_A(\sphere^n)$, the map $\rho$ induces a homeomorphism $\rho':X'\to X$ such that $\p=\rho'\circ\rho$.
\end{itemize}
\end{proposition}
\begin{proof}
With the work done up to this point, the proof of this proposition is the same as in the case of singular Riemannian foliations in spheres, cf. \cite{LR}. We quickly sum up the strategy of the proof.

\begin{itemize}
\item Let $[\cdot]: L^2(\sphere^n)\to L^{2}(\sphere^n)$ be the \emph{averaging operator}, which takes a function $f$ to the function $[f]$ defined by
\[
[f](p)={1\over \textrm{vol}(L_p)}\int_{L_p}fd\textrm{vol}_{L_p},
\]
where $L_p$ is the $\p$-fiber through $p$, and $d\textrm{vol}_{L_p}$ is the volume form induced by the inclusion $L_p\to \sphere^n$.
\item Since the mean curvature of any regular fiber is basic, it follows that $[\cdot]$ takes Lipschitz functions to Lipschitz functions, and $\Delta [f]=[\Delta f]$. By the regularity theory of elliptic equations, it follows that $[\cdot]$ defines a map
\[
[\cdot]:C^\infty(\sphere^n)\to C^\infty(\sphere^n)^\p
\]
where $C^\infty(\sphere^n)^\p$ denotes the set of smooth functions that are constant along the fibers of $\p$, also called \emph{smooth $\p$-basic functions}.
\item The averaging operator extends to a continuous operator $C^{\infty}(\RR^{n+1})\to C^{\infty}(\RR^{n+1})^{C(\p)}$, where $C(\p):\RR^{n+1}\to \textrm{Cone}(X)$ is the manifold submetry taking $t\cdot p$ ($t\in \RR_+$, $p\in \sphere^n$) to $t\cdot \p(p)$, and this operator commutes with rescaling. Therefore, for any homogeneous polynomial $P$, the average $[P]$ is also a homogeneous polynomial, of the same degree of $P$.
\item Let $A=\RR[x_1,\ldots x_{n+1}]^\p$ be the ring generated by $\p$-basic, polynomials. Then by the point above, the averaging operator defines a map $[.]:\RR[x_1,\ldots x_{n+1}]\to A$ which by construction satisfies $[PQ]=P[Q]$ for every $P\in A$, that is, a \emph{Reynolds operator}. By classic work of Hilbert, this implies that $A$ is finitely generated (see also Lemma \ref{L:corReynolds} for a proof).
\item Since $\RR[x_1,\ldots x_{n+1}]$ is dense in $C^{\infty}(\RR^{n+1})$, it follows that $A\subset C^{\infty}(\RR^{n+1})^\p$ is dense as well in the
\todo{(11): ``norm'' to ``topology''}
$C^0$ topology. In particular, the polynomials in $A$ distinguish the fibers of $\p$. In other words, the fibers of $\p$ coincide with the fibers of $\rho$.
\item Letting $\rho_1,\ldots \rho_k$ generators of $A$, $\rho=(\rho_1,\ldots, \rho_k):\sphere^n\to \RR^k$, and $X'=\rho(\sphere^n)$, one can define a map $\rho':X'\to X$ by $\rho'(\rho(x_1,,\ldots x_n))=\p(x_1,\ldots, x_n)$. The function $\rho'$ is well defined and injective because by definition the fibers of $\p_A$ equal the fibers of $\p$. Surjectivity is obvious. Finally, since the $\p$-fibers are compact, the map $\rho'$ is a proper (bijective) map, hence a homeomorphism.
\end{itemize}
\end{proof}

In particular, we get the proof of the first half of Theorem \ref{MT:1-1correspondence}, namely:

\begin{theorem}\label{T:subm-to-poly}
Let $V$ be a Euclidean vector space, and $\p:\sphere(V)\to X$ a $C^2$ manifold submetry. Then the algebra $A=\mathcal{B}(\p)$ of homogeneous $\p$-basic polynomials is a maximal Laplacian algebra, and $\mathcal{L}(A)\sim \p$.
\end{theorem}
\begin{proof}
We start by proving that $A$ is Laplacian. First, being $r^2=\sum_ix_i^2$ constant on the whole sphere, it is $\p$-basic and thus $r^2\in A$. Secondly, let $[\cdot ]:\RR[V]\to A$ be the averaging operator defined in the Proposition \ref{P:homog-pol}, and notice that $[P]=P$ if and only if $P\in A$. Then $\Delta P=\Delta [P]=[\Delta P]$ and thus $\Delta P\in A.$

\todo{(12): word ``the algebra'' added}
By Proposition \ref{P:homog-pol} the algebra $A$ is finitely generated, and letting $\rho_1,\ldots \rho_k$ be generators of $A$, it is clear that two points $p, q\in \sphere(V)$ satisfy $p\sim_A q$ if and only if $\rho_i(p)=\rho_i(q)$ for all $i=1,\ldots k$. In particular, $p\sim_A q$ if and only if $p,q$ are in the same fiber of $\rho:\sphere(V)\to X'$ and thus $\rho\sim \p_A$. Since by Proposition \ref{P:homog-pol} we have $\rho\sim \p$, it follows that $\p\sim \p_A=\mathcal{L}(\mathcal{B}(\p))$.

Finally, we prove that $A$ is maximal. Letting $P\notin A$ a polynomial, it follows by definition of $A$ that there are two points $p, q$ in the same fiber of $\p$, such that $P(p)\neq P(q)$. Since, by the previous point, the fibers of $\p$ coincide with the fibers of $\p_A$, it follows that $f(p)=f(q)$ for any $f\in A$, and thus $A$ is maximal by definition.
\end{proof}

\part{From Laplacian algebras to manifold submetries}

Up to now, we started from manifold submetries and constructed polynomial algebras from them. The goal of this second part is to show that any Laplacian algebra $A\subseteq \RR[V]$ gives rise to a manifold submetry $\hat{\pi}_A:\sphere(V)\to \hat{X}_A$.

\section{Fundamental properties of Laplacian algebras}

In this section, we start exploring the algebraic properties of Laplacian algebras. The main result is that Laplacian algebras admit a Reynolds operator (Theorem \ref{T:Reynolds}). This means that they have many of the same properties as algebras of invariant polynomials, for example being finitely generated.

\subsection{Duality and higher products}\label{SS:duality}
Given a graded polynomial algebra $A$, we will denote by $A_d$ the subspace of degree-$d$ polynomials in A. We define a sequence of symmetric, $\RR$-bilinear products
\begin{align*}
\bullet_k:& \RR[V]_a\otimes \RR[V]_b\to \RR[V]_{a+b-2k}\\
f\bullet_k g&=\sum_{a_1=1}^n\ldots \sum_{a_k=1}^n\left({\partial^k f\over \partial x_{a_1}\ldots \partial x_{a_k}}\right)\left({\partial^k g\over \partial x_{a_1}\ldots \partial x_{a_k}}\right)\\
&= \sum_{|\alpha|=k}{k\choose \alpha}(\partial^\alpha f)(\partial^\alpha g),
\end{align*}

where in the last line $\alpha=(\alpha_1,\ldots \alpha_n)$ is a multi index with $|\alpha|=\sum_i\alpha_i$, ${k\choose \alpha}={k!\over \alpha_1!\ldots \alpha_n!}$ and $\partial^\alpha{f}={\partial^kf\over \partial x_1^{\alpha_1}\ldots \partial x_n^{\alpha_n}}$. The equality between the second and the third line is due to the fact that the number of differentials ${\partial^k \over \partial x_{a_1}\ldots \partial x_{a_k}}$ giving rise to the same differential $\partial^\alpha$, $|\alpha|=k$, is precisely ${k\choose \alpha}$.

Given $f\in \R[V]_k$, $f=\sum_\alpha c_{\alpha}x^\alpha$, define the dual operator $\hat{f}:\R[V]\to \R[V]$  by 
\[
\hat{f}=\sum_\alpha c_\alpha \partial_\alpha.
\]
\todo{(13): We point out that $\widehat{fg}=\hat{f}\circ \hat{g}$}
Since the coefficient $c_\alpha$ are constant, it follows from the definition that $\widehat{fg}=\hat{f}\circ \hat{g}$. It is easy to see from the second definition of $\bullet_k$ that for any polynomial $g$,
\[
{1\over k!}f\bullet_k g=\hat{f}(g),
\]
because both terms are linear in $f$, and it easily holds for monomials.
Observing that $g\bullet_d g$ is a positive constant for every nonzero $g\in \RR[V]_d$ we may define an inner product on each $\R[V]_d$ by
\[
\left< f,g\right>_d=\hat{f}(g)=\hat{g}(f)={1\over d!}f\bullet_d g.
\]

Note that, with respect to this inner product, multiplication by $f$ is adjoint to $\hat{f}$. Indeed,
\begin{equation}
\label{E:adjoint}
\scal{gf,h}_d=\widehat{gf}(h)=\hat{g}(\hat{f}(h))=\scal{g, \hat{f}(h)}_{d-k}.
\end{equation}

\todo{(13+): Added remark about these products in the literature}
\begin{remark}
In this generality, we are not aware of these products being used before, even in classical invariant theory. However, one can easily see that the widely used \emph{polarizations} and \emph{generalized polarizations} of an invariant polynomial $f$ correspond to taking the product $P\bullet_1 f$ and $Q\bullet_2 f$ respectively, for very special choices of polynomials $P, Q$.
\end{remark}

\subsection{Laplacian algebras and Reynolds operators}
The following lemma shows that the products $\bullet_k$ in the previous section, can be defined in terms of the Laplacian.
\begin{lemma}\label{L:bullet-laplacian}
The higher products $\bullet_k$ can be written in terms of the Laplacian and the product structure, via the inductive formula:
\begin{equation}
\label{E:inductive}
f\bullet_0 g:= fg,\qquad f\bullet_{k+1}g:={1\over 2}\big(\Delta(f\bullet_k g)-(\Delta f)\bullet_k g - f\bullet_k(\Delta g)\big)
\end{equation}
\end{lemma}
\begin{proof}
The result is clear for $k=0$. For $k>0$ define ${\bf n}=\{1,\ldots n\}$ and, given $\ai=(a_1,\ldots a_k)\in {\bf n}^k$, let $\partial_\ai f:={\partial^k f\over \partial x_{a_1}\ldots \partial x_{a_k}}$. It is a direct computation that:
\begin{align*}
\Delta(f\bullet_k g)&= \Delta\left(\sum_{\ai\in {\bf n}^k} (\partial_\ai f)(\partial_\ai g)\right)\\
&=2\sum_{\ai\in {\bf n}^{k+1}}  (\partial_\ai f)( \partial_\ai g)+ \sum_{\ai\in{\bf n}^k} (\Delta \partial_\ai f) (\partial_\ai g)+ \sum_{\ai\in{\bf n}^k}  (\partial_\ai f)( \Delta\partial_\ai g)
\end{align*}
Since $\Delta \partial_\ai f= \partial_\ai \Delta f$ and same for $g$, the computations become
\begin{align*}
\Delta(f\bullet_k g)&=2f\bullet_{k+1} g+ \sum_{\ai\in{\bf n}^k} ( \partial_\ai \Delta f) (\partial_\ai g)+ \sum_{\ai\in{\bf n}^k}  (\partial_\ai f)( \partial_\ai \Delta g)\\
&=2f\bullet_{k+1} g+  ( \Delta f)\bullet_k g+ f\bullet_k(\Delta g)
\end{align*}
and the result is proved.
\end{proof}

\begin{corollary}
Let $A$ be a Laplacian algebra. Then:
\begin{enumerate}
\item For any $f,g\in A$, and any $k$, $f\bullet_k g\in A$.
\item $A$ is a graded ring.
\item For any $f\in A$, the operator $\hat{f}$ takes $A$ into $A$.
\end{enumerate}
\end{corollary}
\begin{proof}
1. Follows directly from Lemma \ref{L:bullet-laplacian}, since the operations $\bullet_k$ are defined in terms of the algebra structure, and the Laplacian.

2. Decompose $f\in A$ into its homogeneous parts $f=\sum_j f_j$, where $f_j$ has degree $j$. Then $\frac{1}{2} r^2\bullet_1 f =\sum_j j \,f_j\in A$. Applying this $\deg(f)$ many times and using the invertibility of the Vandermonde matrix shows that $f_j\in A$ for every $j$.

3. For $f$ homogeneous it is clear, since $\hat{f}(g)=f\bullet_j g$, with $j=\deg(f)$. In general, decompose $f\in A$ into its homogeneous parts $f=\sum_j f_j$, where $f_j$ has degree $j$. By the previous point, $f_j\in A$ for all $j$. Then $\hat{f}=\sum \hat{f}_j$ and each $\hat{f}_j$ takes $A$ into $A$.
\end{proof}

We can now prove the existence of the Reynolds operator (Theorem  \ref{T:Reynolds} below). To do this, let us define the projection $\Pi:\RR[V]\to A$ degree wise, by letting $\Pi_d:\RR[V]_d\to A_d$ be the orthogonal projection with respect to the inner product $\scal{\cdot\,,\,\cdot}$ defined in Section \ref{SS:duality}. Recall that, if $A$ is the algebra of homogeneous basic polynomials of a manifold submetry $\p:\sphere(V)\to X$, then there is an averaging operator $[\cdot]:\RR[V]\to A$, see the proof of Proposition \ref{P:homog-pol}. 

Then:

\begin{theorem}
\label{T:Reynolds}
Let $A\subset \R[V]$ be a Laplacian algebra. Then the projection $\Pi=\bigoplus_d \Pi_d:\R[V]\to A$ is a Reynolds operator, that is,  $\Pi(fg)=f\Pi(g)$ for $f\in A$ and $g\in\R[V]$. Moreover, if $A$ is the algebra of basic polynomials of a manifold submetry, then $\Pi$ coincides with the averaging operator.
\end{theorem}
\begin{proof}
Let $f\in A_k$ and $g\in \R[V]_{d-k}$. Let $g=g_1+g_2$, where $g_1=\Pi_{d-k}(g)$ lies in $A_{d-k}$ and $g_2$ is orthogonal to $A_{d-k}$. By linearity,
\[\Pi_d(fg)=\Pi_d(fg_1)+\Pi_d(fg_2)=fg_1+\Pi_d(fg_2)\]
and therefore it suffices to show that $\Pi_d(fg_2)=0$. But this is true because, for every $Q\in A_{d}$, $\scal{Q,fg_2}=\scal{\hat{f}Q,g_2}$ (by \eqref{E:adjoint}), which is zero since $\hat{f}Q\in A_{d-k}$.

Now assume $A$ is the algebra of homogeneous basic polynomials of a manifold submetry $\sphere(V)\to X$. Since the averaging operator $[\cdot]$ and the Reynolds operator $\Pi$ are idempotent with the same image $A$,  showing that they coincide is equivalent to showing that the kernel of $[\cdot]$ is orthogonal to $A$. So let $g\in \R[V]_d$ such that $[g]=0$, and let $f\in A_d$.
Since the Laplacian and the
\todo{(14): typo fixed}
averaging operator commute, $\Delta[P]=[\Delta P]$ for any $P\in \RR[V]$ (cf. the proof Proposition \ref{P:homog-pol}) and the inductive formula \eqref{E:inductive} for $\bullet_d$ implies that  $[f\bullet_d g] =f\bullet_d [g]$. Therefore
\[ \scal{f,g}=f\bullet_d g= [f\bullet_d g] =f\bullet_d [g]=0\]
because $f\bullet_d g$ is a constant, and hence basic.
\end{proof}

The existence of a Reynolds operator is crucial in Invariant Theory, and we collect below a few standard consequences which we will need later:
\begin{lemma}
\label{L:corReynolds}
Let $A\subset \R[V]$ be a Laplacian algebra. Then
\begin{enumerate}[a)]
\item $A$ is finitely generated.
\todo{(20), part 1: modified proposition by adding new item (b), and deleting old item (d). Comment (16) is now referring to a part that got deleted}
\item Let $F(A)$ be the field of fractions of $A$. Then $A=F(A)\cap \R[V]$.
\item $A$ is integrally closed in its field of fractions.
\end{enumerate}
\end{lemma}
\begin{proof}
\begin{enumerate}[a)]
\item Let $A^+\subseteq A$ be the subspace generated by the homogeneous polynomials of positive degree, and let $I$ be the ideal in $\RR[V]$ generated by $A^+$. Since $\RR[V]$ is Noetherian, $I=(\rho_1,\ldots,\rho_k)$ for some $\rho_1,\ldots \rho_k\in A^+$. We claim that $\rho_1,\ldots \rho_k$ generate $A$ as a ring, by induction on the degree. Suppose that they generate $A^{<d}$, and let $f\in A_d$. Since $f\in I$ we can write $f=\sum a_i \rho_i$, where $a_i\in \RR[V]$ can be chosen homogeneous, of degree $\deg(f)-\deg(\rho_i)<d$. Since $f$ and $\rho_i$ belong to $A$, we can apply $\Pi$ to the equation and obtain
\[
f=\Pi(f)=\Pi\left(\sum a_i \rho_i\right)=\sum \Pi(a_i)\rho_i
\]
Since $\Pi(a_i)$ live in $A^{<d}$, by the induction hypothesis they can be written as polynomials in the $\rho_i$'s, and therefore so can $f$. This proves the induction step.

\item Let $\RR(V)$ be the field of fractions of $\RR[V]$. Since $A\subseteq \RR[V]$, clearly $F(A)\subseteq \RR(V)$. Let $f,g\in A$ and $h\in \R[V]$ so that ${f\over g}=h\in F(A)\cap \R[V]$. Then $f=hg$ and applying the Reynolds operator we get $f=g[h]$. Therefore, ${f\over g}=[h]\in A$.

\item Suppose that $\alpha={f\over g}\in F(A)$ is a root of a monic polynomial $P(t)=t^n+\sum h_i t^{n-i}$ in $A[t]$. Then in particular $\alpha\in \RR(V)$ and $P\in \RR[V][t]$. Since $\RR[V]$ is a Unique Factorization Domain,
\todo{(15):``in particular'' removed}
it is integrally closed in its field of fraction, and thus $\alpha\in \RR[V]$. Hence $\alpha\in F(A)\cap \R[V]$ and by the previous point $\alpha\in A$.

%
\end{enumerate}
\end{proof}

\section{Laplacian algebras give rise to submetries}\label{S:lap-alg-subm}

The main goal of the next two sections is to prove the following:
\begin{theorem}
\label{T:sms construction}
Let $A\subset \RR[V]$ be a  Laplacian algebra. Then:
\begin{enumerate}[a)]
\item There exists a \sms $\hat\p_A:\sphere(V)\to \hat X$ whose fibers coincide with the level sets of $A$, on an open and dense set.
\item If furthermore $A$ is maximal, then \emph{all} fibers of $\hat \p_A$ coincide with the level sets of $A$.
\end{enumerate}
\end{theorem}

Let $A\In \R[V]$ denote a  Laplacian algebra, which for the moment is not necessarily maximal. The strategy is to produce a manifold submetry $\p$ from the whole of $V$ to a cone $X=C(Y)$ such that the preimage of the vertex in $C(X)$ is the origin in $V$. Then by equidistance, it follows that $\p$ restricts to the manifold submetry $\p:\sphere(V)\to Y$ we are looking for. In this section, we produce the submetry, and in the next section we prove that the fibers are smooth.

\subsection{Riemannian submersion almost everywhere}\label{SS:riem-subm}

Let $A\subset \R[V]$ be a Laplacian algebra. By Lemma \ref{L:corReynolds}, $A$ is finitely generated, so let $\rho_1,\ldots \rho_k$ be homogeneous generators of $A$, and let $\rho:V\to \R^k$ be the map $\rho(x)=(\rho_1(x),\ldots \rho_k(x))$. 

Let $V^{\reg}$ be the open dense set of $V$ where the rank of $d\rho$ (which equals the dimension of $\textrm{span}(\nabla \rho_1,\ldots \nabla \rho_k)$) is maximal, let $m$ denote such a maximal rank, and denote $V^{\sing}$ the complement of $V^{\reg}$. The set $V^{\reg}$ can be equivalently defined as the set where the matrix $\hat{B}\in \textrm{Sym}^2(A^k)$ given by $\hat{B}_{ij}=\rho_i\bullet_1 \rho_j$ 
\todo{(17): typo fixed.}
has maximal rank (this is because $\hat{B}=(d\rho)\cdot (d\rho)^*$). Because $A$ is Laplacian, the entries of $\hat{B}$ are in $A$, and in particular $V^{\reg}$ is a union of level sets of $\rho$. Moreover, by the Inverse Function Theorem, the restriction of $\rho$ to $V^{\reg}$ is a submersion onto the image. Our first result is:

\begin{proposition}[Riemannian submersion almost everywhere]
\label{P:regular}
The restriction of $\rho$ to $V^{\reg}$ is a Riemannian submersion, for an appropriate choice of metric on $\rho(M^{\reg})$.
\end{proposition}
\begin{proof}
From the Inverse Function Theorem, the leaves in $V^{\reg}$ are smooth, and with the same dimension. Moreover, since $\rho_*$ has constant rank at all points in $V^{\reg}$, the image $X^{\reg}=\rho(V^{\reg})$ is a smooth manifold as well, and the map $\rho:V^{\reg}\to X^{\reg}$ is a submersion. We need to prove that there exists a metric in $X^{\reg}$ such that $\rho$ becomes a Riemannian submersion. To produce such a metric, consider the vector fields $X_i=\rho_*(\nabla \rho_i)$ in $X^{\reg}$. Given the standard basis $e_i$ of $\R^k$, we can write $X_i(\rho(p))=\sum_jb_{ij}(\rho(p)) e_j$, where
\[
b_{ij}(\rho(p))=\scal{\nabla \rho_i,\nabla \rho_j}_p=\rho_i\bullet_1\rho_j(p)=\hat{B}_{ij}(p)
\]
 (recall, the entries of $\hat B_{ij}$ belong to $A$ hence can be written as polynomials in $\rho_1,\ldots, \rho_k$).

For indices $1\leq i_1<\ldots< i_m\leq k$  (recall that $m$ is the rank of $d\rho$), let $U_{\{i_1,\ldots i_m\}}\subseteq V^{\reg}$ be  the open set where $X_{i_1},\ldots X_{i_m}$  are linearly independent. For sake of notation let us consider $U_{\{1,\ldots m\}}$. In this case, the matrix $B=(b_{ij})_{i,j=1,\ldots m}$ is nondegenerate and positive definite. On $\rho(U_{\{1,\ldots m\}})$, define the metric
\[
b(X_i,X_j)= b_{ij},\quad \forall i,j=1,\ldots m.
\]
Then,  $\rho$ restricted to $U_{\{1,\ldots, m\}}$ is a Riemannian submersion. Moreover, covering $X^{\reg}$ by open sets of the form $\rho(U_{\{i_1,\ldots i_m\}})$, the metric can be extended on the whole of $X^{\reg}$, and thus $\rho$ is a Riemannian submersion.
\end{proof}

\begin{proposition}\label{P:reg-equidistant}
For any $p_*, q_*\in X^{\reg}=\rho(V^{\reg})$, the fibers $\rho^{-1}(p_*)$ and $\rho^{-1}(q_*)$ are equidistant.
\end{proposition}
\begin{proof}
Fixing $p_*, q_*\in X^{\reg}$, let $p_1, p_2\in \rho^{-1}(p_*)$. To prove that $\rho^{-1}(p_*)$ and $\rho^{-1}(q_*)$ are equidistant, it is enough to show that $d(p_1, \rho^{-1}(q_*))=d(p_2, \rho^{-1}(q_*))$. Let $\gamma:[0,\ell]\to V$, $\gamma(t)=p_1+tv$ be a shortest geodesic from $p_1$ to $\rho^{-1}(q_*)$. This geodesic may in principle leave $V^{\reg}$ at some points, but since $V^{\sing}$ is algebraic and $\gamma$ is an algebraic map, it follows that $\gamma(t)\in V^{\reg}$ for all but discretely many $t\in [0,\ell]$. Furthermore, by the first variation of length it follows that $v=\gamma'(\ell)$ is horizontal at $t=\ell$. Since $\rho$ is a Riemannian submersion around $\rho^{-1}(q_*)$ it follows that $v=\gamma'(t)$ is horizontal around $t=\ell$, that is, $v$ is a linear combination of $\nabla \rho_1(\gamma(t)), \ldots, \nabla \rho_k(\gamma(t))$ for all $t$ in a neighborhood of $\ell$ in $[0,\ell]$. However, this is an algebraic condition, thus it holds for all $t\in [0,\ell]$, and in particular $\gamma(t)$ is horizontal around $t=0$.

Write $v=\sum_i a_i\nabla\rho_i(p_1)$, and define $v_2=\sum_i a_i\nabla\rho_i(p_2)$, $\gamma_2(t):=p_2+tv_2$. By construction, $\gamma_2$ is a horizontal geodesic which projects to the same geodesic in $X^{\reg}$ as $\gamma(t)$, for all $t$ small enough. Then the two polynomial maps $P_1,P_2:[0,\ell]\to \RR^k$ given by $P_1(t)=\rho(\gamma(t))$, $P_2(t)=\rho(\gamma_2(t))$ coincide in a neighborhood of $0\in [0,\ell]$, and thus they coincide everywhere. In particular, $\gamma_2$ is a geodesics from $p_2$ to $\rho^{-1}(q_*)$, of the same length as $\gamma$, and therefore $d(p_2,\rho^{-1}(q_*))\leq d(p_1,\rho^{-1}(q_*))$. By inverting the roles of $p_1$ and $p_2$, the other inequality follows, and thus $d(p_2,\rho^{-1}(q_*))= d(p_1,\rho^{-1}(q_*))$.
\end{proof}

\subsection{Submetry everywhere}

By Proposition \ref{P:regular}, a Laplacian algebra $A\subseteq \RR[V]$ produces a Riemannian submersion $\rho^{\reg}:=\rho|_{V^{\reg}}:V^{\reg}\to X^{\reg}$, on an open dense set $V^{\reg}$ of $V$. We want to extend $\rho^{\reg}$ to a manifold submetry $\hat{\rho}$ defined on the whole of $V$. We start with showing that $\rho^{\reg}$ can be extended to a submetry.

\begin{proposition}\label{P:submetry}
There is a metric space $\hat{X}$ containing $X^{\reg}$, and a submetry $\hat{\rho}:V\to \hat{X}$ extending $\rho^{\reg}$.
\end{proposition}

\begin{proof}
On $X^{\reg}$, define the distance function by $d(p_*,q_*)=d_{V}(\rho^{-1}(p_*), \rho^{-1}(q_*))$. Since by Proposition \ref{P:regular} the regular fibers of $\rho$ are equidistant,
this is indeed a distance function. Define $\hat{X}$ as the metric completion of $(X^{\reg},d)$. Then we can extend $\rho^{\reg}$ to $\hat{\rho}:V\to \hat{X}$ by defining, for $p\in V$ given as a limit of a sequence $\{p_i\}_i$ in $V^{\reg}$, $\hat{\rho}(p)=\lim_{i\to \infty} {\rho}(p_i)$ where $p_i$ is a sequence of points in $V^{\reg}$ converging to $p$.

First, we claim that $\hat{\rho}$ is well defined. In fact, if $\{p_i^1\}_i$ and $\{p_j^2\}_j$ are two sequences converging to $p$ then $d_V(p_i^1,p_2^i)\to 0$, therefore $d_{\hat{X}}(\rho(p_i^1), \rho(p_i^2))\to 0$, and by definition of metric completion the two sequences $\rho(p_i^1)$, $\rho(p_i^2)$ define the same limit point. By definition, $\hat{\rho}$ is continuous.

Secondly, we claim that $\hat{\rho}$ is a submetry. Clearly it is distance non-increasing, since it is the completion of $\rho^{\reg}$ and this is distance non-increasing. We thus need to prove that for any $p\in V$ and any $r>0$, $B_r(\hat{\rho}(p))\subseteq \hat{\rho}(B_r(p))$. Let $q_*\in B_r(\hat{\rho}(p))$ and consider sequences $\{q_*^i\}_i\subset X^{\reg}$ converging to $q_*$, $\{p_i\}_i\subset V^{\reg}$ converging to $p$, and pick points $q_i\in \rho^{-1}(q_*^i)$ such that $d(q_i, p_i)=d(q_*^i,\rho(p_i))$. The existence of such points $q_i$ is assured by the fact that the $\rho$-fibers in $V^{\reg}$ are equidistant. Since the points $q_i$ are contained in a ball around $p$, there is a subsequence (which we still denote by $q_i$) converging to some $q\in V$. By construction, $\hat\rho(q)=q_*$ and
\[
d(q,p)=\lim_{i\to\infty} d(q_i, p_i)= \lim_{i\to\infty}d(q_*^i,\rho(p_i))=d(q_*, \hat\rho(p))<r
\]
therefore $q\in B_r(p)$ and thus $q_*\in \hat{\rho}(B_r(p))$.
\end{proof}

\section{Laplacian algebras give rise to manifold submetries}\label{S:algebras-to-smss}

The goal of this section is to show that the submetry $\hat{\rho}:V\to \hat{X}$ defined in the previous section is in fact a manifold submetry, thus finishing the proof of Theorem \ref{T:sms construction}(b). For this, we need to show that each singular fiber \todo[color=cyan]{typo fixed} of $\hat{\rho}$ is a smooth embedded submanifold (all of whose connected components have the same dimension). This will be done in three steps: First, using the transverse Jacobi field equation (introduced in \cite{Wilking07}) we will show that $L$ is a disjoint union of smooth immersed submanifolds. Second, we will show that $L$ has positive reach, which implies that $L$ is a disjoint union of smooth embedded submanifolds. Third, we will show that the connected components of $L$ have the same dimension.

\begin{proposition}\label{P:sing-immersed}
For any singular fiber $L'$ of $\hat{\rho}:V\to \hat{X}$, there is a regular fiber $L$ and a differentiable map $\phi: L\to V$ with locally constant rank and $\phi(L)=L'$.
\end{proposition}
\begin{proof}
Fixing a singular fiber $L'$ and a point $q\in L'$, take any regular leaf $L$, let $\gamma:[0, 1]\to V$, be a minimizing geodesic from $L$ to $q$, and let $p:=\gamma(0)$. Up to substituting $L$ with a regular fiber through a later time $\gamma(t)$, we can suppose that all fibers through $\gamma(t)$, $t\in(0,1)$ are regular. Then $\gamma'(0)$ is perpendicular to $L$ at $0$, thus $\gamma'(0)=\sum a_i \nabla\rho_i(p)$ for some constants $a_i$, and we can define the normal vector field $X=\sum_i a_i \nabla \rho_i$ along $L$, the map $\Phi:L\times \RR\to V$ by $\Phi_t(p')=p'+t X(p')$, and the map $\phi=\Phi_1$.

We first claim that $\phi(L)=L'$. On the one hand, the geodesics $\gamma_p(t):=\Phi_t(p)$ all project to the same geodesic in $X^{\reg}$ near $t=0$, then they meet the same geodesics for all $t$ (see the proof of Proposition \ref{P:reg-equidistant}) and therefore $\phi(L)\subseteq L'$. On the other hand, since $\hat{\rho}$ is a submetry and $d(\hat{\rho}(\Phi_t(L)), \hat{\rho}(L'))\to 0$ as $t\to 1$, for any $q'\in L'$ there is a sequence of times $t_i\to 1$ and points points $p_i\in \Phi_{t_i}(L)$ converging to $q'$. By the continuity of $\Phi$, it follows that $q'\in \Phi_1(L)=\phi(L)$ and thus $\phi(L)=L'$.
\\

We are left to prove that $\phi$ has locally constant rank. Equivalently, we can prove that $\ker d\phi$ is locally constant. For every $p\in L$, define $\gamma_p(t)=\Phi_t(p)$, and $W_p$ the space of Jacobi fields $J_v(t)=d_{\gamma_p(t)}\Phi_t(v)$, for $v\in T_pL$. Notice that these really are Jacobi fields, since they can be written also as $J_v(t)={d\over ds}\big|_{s=0}\gamma_{\alpha(s)}(t)$, where $\alpha$ is a curve in $L$ with $\alpha'(0)=v$. Furthermore, for any $J_1, J_2\in W_p$ and any $t\in (0,1)$ we have $J_i'(t)=S_{\gamma_p'(t)}J_i(t)$ where $S_{\gamma_p'(t)}$ is the shape operator of $\Phi_t(L)$, and thus
\[
\scal{J_1'(t),J_2(t)}-\scal{J_1(t), J_2'(t)}=\scal{S_{\gamma_p'(t)}J_1(t),J_2(t)}-\scal{J_1(t), S_{\gamma_p'(t)}J_2(t)}=0
\]
It follows that $W_p$ is an isotropic space (see Appendix \ref{A:Lagr}). Furthermore,  by construction the focal function $f_{W_p}(t)$ is zero for $t\in (0,1)$ and equal to $\dim\ker d_p\phi$ for $t=1$.

For any $p\in L$, the space $W_p$ can be extended to a Lagrangian space of Jacobi fields
\[
\Lambda_p=W_p\oplus \{J\mid J(0)=0, J'(0)\perp T_pL\oplus \gamma_p'(0)\},
\]
which corresponds to the space of normal Jacobi fields along $\gamma_p$, given as variations by horizontal geodesics through $L$. 

Set $\epsilon$ small enough, that the fibers of $\hat{\rho}$ through $\gamma_p(t)$ have constant dimension in $(1, 1+\epsilon)$, and set the function $\iota:L\to \RR$ given by $\iota(p)=\ind_{[0,1+\epsilon]}\Lambda_p$. By Equation \eqref{E:focal-sum}, we have
\[
\iota(p)=\iota^v(p)+\iota^h(p),\qquad \iota^v(p):=\ind_{[0,1+\epsilon]}W_p,\quad \iota^h(p):=\ind_{[0,1+\epsilon]}\Lambda_p/W_p.
\]
We call $\iota^v$ the \emph{vertical index} and $\iota^h$ the \emph{horizontal index}. By the discussion above, $\iota^v(p)=\dim\ker\phi$, so in order to prove the final claim it is enough to prove that $\iota-\iota^h$ is locally constant on the fibers of $\hat{\rho}$.

On the one hand, since $\iota$ denotes the index of a Lagrangian space, it follows from Proposition \ref{P:continuity} in the Appendix \ref{A:Lagr}, that this function is locally constant. On the other hand, for any $p\in L$ the Lagrangian space $\Lambda_p/W_p$ can be identified with the (isotropic) space of Jacobi fields along $\gamma_*|_{[0,1+\epsilon]\setminus\{1\}}$ in $X^{\reg}$ which vanish at $0$ (See Example \ref{E:example}). In particular, $\iota^h(p)$ does not depend on $p\in L$.
%
%
%
\end{proof}

The second step it to show that each connected component of a singular fiber $L'$ is in fact embedded.

\begin{proposition}\label{P:sing-embedded}
For any singular fiber $L'$ of $\hat{\rho}$, and any $p\in L'$, there is a neighborhood $U_p$ of $p$ in $V$ such that $U_p\cap L'$ is a smooth manifold.
\end{proposition}

\begin{proof}
\todo{(18): Proof simplified}
Fix a singular fiber $L'$. 
By Proposition \ref{P:sing-immersed}, every connected component of $L'$ is an immersed submanifold, thus the tangent space at every point is a union of vector spaces.
 On the other hand, since $\hat{\rho}$ is a submetry, by Proposition 12.10 of \cite{Lyt} every fiber has \emph{positive reach}, that is, for every $p\in L'$ and every $\epsilon$ small enough, there is a map $Upt: B_{\epsilon}(p)\to L'$ such that $Upt(q)$ is the unique point in $L'$ minimizing the distance between $q$ and $L'$. It is well known (cf. \cite{Federer}, Part (12) of Theorem 4.8) that a set of positive reach has a tangent space at each point, which is a convex cone. In particular, each tangent space of $L'$ consists of a single vector space. By Proposition 1.4 in \cite{Lytchak2005} it follows that $L'$ is an injectively immersed $C^{1,1}$ manifold. Since $L'$ is also a closed immersed smooth manifold, it follows that it is embedded as well.

\end{proof}

It remains to prove that different connected components of the same fiber have the same dimension.

\begin{lemma}\label{L:H-basic}
The mean curvature $H$ of the regular fibers of $\hat{\rho}:V\to \hat{X}$ descends to a vector field on $X^{\reg}$.
\end{lemma}
\begin{proof}
It is enough to show that for every $f\in A$, $\scal{H, \nabla f}$ is constant along the fibers of $\rho^{\reg}:V^{\reg}\to X^{\reg}$. Since $\rho^{\reg}$ is a Riemannian submersion and $f$ is constant along its fibers, there is a smooth function $\underline{f}\in C^\infty(X^{\reg})$ such that $f=\underline{f}\circ \rho$. Then straightforward computations (cf. \cite{AR}) show that
\[
\Delta f= (\Delta_{X^{\reg}} \underline{f})\circ \rho +\scal{H, \nabla f}.
\]
Since $A$ is Laplacian, $\Delta f$ is also constant along $L$, and therefore so is $\scal{H, \nabla f}$.
\end{proof}

\begin{proposition}\label{P:conn-comp-same-dim}
Any two connected components of a same fiber of $\hat{\rho}:V\to \hat{X}$ have the same dimension.
\end{proposition}
\begin{proof}
This is clearly true for fibers in $V^{\reg}$, thus we focus on the singular fibers. 

By Theorem 10.1 in \cite{Lyt}, it follows that the submetry $\hat{\rho}:V\to \hat{X}$ factors as $V\stackrel{\hat{\rho}_0}{\longrightarrow} \hat{X}_0\to \hat{X}$ where the fibers of $\hat{\rho}_0$ are the connected components of the fibers of $\hat \rho$, and $\hat{X}_0\to \hat{X}$ is a submetry with discrete fibers. By Proposition \ref{P:sing-embedded}, the submetry $\hat{\rho}_0$ is in fact a manifold submetry.

Let $p_1, p_2$ be points lying in different connected components of a singular fiber $L'=\hat{\rho}^{-1}(p_*)$, and let $L'_1, L'_2 \subseteq L'$ the fibers of $\hat{\rho}_0$ containing $p_1$ and $p_2$, respectively. Since $\hat{\rho}_0$ is a manifold submetry, it follows from Lemma \ref{L:Qgeod-uniq.} that there are horizontal geodesics $\gamma_1, \gamma_2:[0,\ell]\to V$ such that $\hat\rho(\gamma_1)=\hat\rho(\gamma_2)$, $\gamma_i|_{[0,\ell)}\subset V^{\reg}$, and $\gamma_i(\ell)=p_i$, $i=1,2$.

By Proposition \ref{P:Jacfields}, there are families of Jacobi fields $W_1, W_2$ along $\gamma_1$ and $\gamma_2$ respectively, such that $W_i(t)=\{J(t)\mid J\in W_i\}$ is the tangent space to the fiber (of $\hat{\rho}$ or $\hat{\rho}_0$, it is the same) through $\gamma_i(t)$. Therefore, it is enough to prove that $\dim W_1(\ell)=\dim W_2(\ell)$.

Recall that $W_i$ are isotropic subspaces (cf. Appendix), and therefore for every $t\in [0,\ell)$, $\dim W_i(t)=\dim W_i=\dim V - m$, where $m$ denotes the rank of $\rho^{\reg}$. Furthermore, for every $t\in [0,\ell)$ there is a symmetric endomorphism $S_i(t):W_i(t)\to W_i(t)$ such that $S_i(t)J(t)=pr_{W_i(t)}J'(t)$ for every $J\in W_i$, where $pr_{W_i(t)}$ denotes the projection onto $W_i(t)$. This endomorphism coincides with the shape operator of the leaf through $\gamma_i(t)$, in the direction of $\gamma_i'(t)$, and it satisfies the Riccati equation
\[
S'_i(t)+S_i^2(t)=0,
\]
where $S'_i(t):W(t)\to W(t)$ is the covariant derivative of $S_i(t)$. By standard theory of solutions to the Riccati equation
\todo{(19): added reference for this fact}
(cf. Remark 1, and Proposition of \cite{Heintze1990}), close to $t=\ell$ the operator $S_i(t)$ becomes asymptotic to
\[
S_i(t)\sim \left(\begin{array}{cc}{1\over \ell-t} I_{d_i} &  \\   & \tilde{S}_i(t)\end{array}\right)
\]
where $d_i=\dim\{J\in W_i\mid J(\ell)=0\}=\dim W_i-\dim W_i(\ell)$, and $\tilde{S}_i(t)$ is bounded as $t\to \ell^-$. In particular, close to $t=\ell$ we have
\[
\scal{H(\gamma_i(t)),\gamma_i'(t)}=\operatorname{tr}(S_i(t))= {d_i\over \ell-t}+O(1)
\]
On the other hand, since $\gamma_1, \gamma_2$ project to the same geodesic in $X^{\reg}$ and, by Lemma \ref{L:H-basic}, $H$ projects to a vector field in $X^{\reg}$, it follows that $\scal{H(\gamma_1(t)),\gamma_1'(t)}=\scal{H(\gamma_2(t)),\gamma_2'(t)}$ and thus $d_1=d_2$. Since $\dim L_i'=\dim W_i(\ell)=n-m-d_i$, we have the result.
\end{proof}

By collecting the results in the previous section and this one, we obtain a proof of Theorem \ref{T:sms construction}.

\begin{proof}[Proof of Theorem \ref{T:sms construction}]
a) Given a Laplacian algebra $A\subseteq \RR[V]$, by Corollary \ref{L:corReynolds} there are finitely many functions $\rho_1,\ldots, \rho_k$ generating $A$. Let $\rho=(\rho_1,\ldots \rho_k):V\to \RR^k$, and define the submetry $\hat{\rho}:V\to \hat{X}$ as in Proposition \ref{P:submetry}. By Proposition \ref{P:sing-embedded}, the fibers of $\hat{\rho}$ are unions of smoothly embedded submanifolds, and by Proposition \ref{P:conn-comp-same-dim} the connected components of each fiber have the same dimension. Therefore, $\hat{\rho}$ is a manifold submetry. Furthermore, since $r^2\in A$, it follows in particular that the origin is a (0-dimensional) fiber of $\hat{\rho}$, and the other fibers are contained in the distance spheres of $V$ around the origin. In particular, the restriction of $\hat{\rho}$ to $\sphere(V)$ defines a manifold submetry
\[
\hat{\p}_A=\hat{\rho}|_{\sphere(V)}:\sphere(A)\to \hat{X}_A:=\hat{\rho}(\sphere(V))\subset \hat{X}.
\]
Since $\mathcal{L}(A)$ is equivalent to $\rho\big|_{\sphere(V)}$, in particular its restriction to $V^{\reg}\cap \sphere(V)$ is equivalent to  $\hat{\p}_A$.

b) Suppose now that $A$ is also maximal, and thus $A=\mathcal{B}(\mathcal{L}(A))$. Since every $f\in A$ is, by construction, constant along the fibers of $\hat{\p}_A$, it follows that the $\hat{\p}_A$-fibers are contained in the fibers of $\mathcal{L}(A)$, and $\hat{A}:=\mathcal{B}(\hat{\p}_A)$ contains $\mathcal{B}(\mathcal{L}(A))=A$.

By Theorem \ref{T:subm-to-poly}, we have $\hat{\p}_A\sim \mathcal{L}(\mathcal{B}(\hat{\p}_A))=\mathcal{L}(\hat{A})$ and thus, in order to show that $\hat{\p}_A\sim \mathcal{L}(A)$, it is enough to prove that $\hat{A}=A$.

We start by proving that $A$ and $\hat{A}$ have the same field of fractions: $F(A)=F(\hat{A})$. Clearly since $A\subset \hat{A}$, $F(A)\subset F(\hat{A})$ and it is enough to prove the other inclusion. Let $f\in \hat{A}$, and let $g\in A$ be a nonzero polynomial vanishing on $V^{\sing}$ -- for example, take $P$ the be the product of all the determinants of the $m\times m$ minors of $B=(\rho_i\bullet_1 \rho_j)_{i,j=1,\ldots k}$ (see Section \ref{SS:riem-subm}). Then the product $fg$ is zero on $V^{\sing}$, and on $V^{\reg}$ it is constant along the fibers of $\mathcal{L}(A)$. Thus $fg=h\in \mathcal{B}(\mathcal{L}(A))= A$, and $f={h\over g}\in F(A)$. This gives $\hat{A}\subseteq F(A)$ and thus $F(\hat{A})\subseteq F(A)$.
\todo{(20): part 2, changed the proposition according to the suggestion. We also modified Lemma \ref{L:corReynolds} consistently with this.}
By Lemma \ref{L:corReynolds} part (b), since both $\hat{A}$ and $A$ are Laplacian, it follows that
\[
\hat{A}=F(\hat{A})\cap \R[V]=F(A)\cap \R[V]=A.
\]
\end{proof}

\todo{(20+): added remark here}
\begin{remark}
Assume $A\subseteq \RR[V]$ is a Laplacian but not necessarily maximal algebra. Then by Theorem \ref{T:sms construction} there exists a \sms $\hat{\sigma}_A:\sphere(V)\to \hat{X}$ and the algebra $\hat{A}=\mathcal{B}(\hat{\sigma}_A)$ is a maximal Laplacian algebra containing $A$ since, by construction of $\hat{\sigma}_A$, all the polynomials of $A$ are constant along the $\hat{\sigma}_A$-fibers. Again by construction, it also follows that $\hat{\sigma}_A=:\mathcal{L}(\hat{A})$ coincides with $\mathcal{L}(A)$ on the open dense set $\sphere(V^{\reg})$. By the proof of Theorem \ref{T:sms construction}, in order to prove that $A=\hat{A}$ (hence show that $A$ is, after all, maximal), it would be enough to show that $F(A)=F(\hat{A})$.
\end{remark}

\begin{proof}[Proof of Theorem \ref{MT:1-1correspondence}]
Given a manifold submetry $\p:\sphere(V)\to X$, it follows from Theorem \ref{T:subm-to-poly} that $\mathcal{B}(\p)\subseteq\RR[V]$ is a maximal and Laplacian algebra, and $\mathcal{L}(\mathcal{B}(\p))\sim \p$. Letting $A$ be a maximal and Laplacian algebra, it follows from Theorem \ref{T:sms construction} that $\mathcal{L}(A)\sim \hat{\p}_A$ for some manifold submetry $\hat{\p}_A:\sphere(V)\to \hat{X}_A$ and, since $A$ is maximal, $\mathcal{B}(\mathcal{L}(A))=A$.
\end{proof}

\part{Disconnected fibers and the maximality conjecture}
\section{Disconnected fibers}\label{S:disc-fibers}
In this section, we study submetries $\p:\sphere(V)\to X$ with disconnected leaves. In particular, we prove Theorem \ref{MT:connected-vs-integral} and Theorem  \ref{MT:disconnected}.

By \cite{Lyt}, any submetry $\p:\sphere(V)\to X$ factors as $\sphere(V)\stackrel{\p_c}{\to}X_c\to X$, where $\pi:X_c\to X$ is a submetry with finite fibers, and the fibers of $\p_c$ are the connected components of the fibers of $\p$.

Recall from Appendix \ref{A:manif-subm} that any manifold submetry induces a stratification by the dimension of the fibers. In our case, $\p$ and $\p_c$ induce the same stratification, and we let $\sphere(V)^{(2)}$ be the union of the strata $\Sigma_p$ of codimension $\leq 2$ (see Section \ref{SS:HTL}).  Since the complement of $\sphere(V)^{(2)}$ in $\sphere(V)$ consists of finitely many submanifolds of codimension $\geq 3$, it follows by transversality that $\sphere(V)^{(2)}$ is simply connected.
Since $\p_c$ is a manifold submetry with connected fibers, we can apply Proposition \ref{P:orbifold-part-srf} in Appendix \ref{A:manif-subm}, which says that the partition $(\sphere(V)^{(2)},\mathcal{F})$ into the fibers of $\p_c$ is a singular Riemannian foliation.
\todo{(21b): part 1: restructured the beginning of this section according to suggestion 21. A number of definition suppressed since they are unnecessary now.}

\todo{(21b): part 2: After the suggestion of the referee, it was possible to incorporate old Proposition 32 with the proof of theorem D.}
We are finally able to prove the main results for this section.

\begin{proof}[Proof of Theorem \ref{MT:disconnected}]
Since $\sphere(V)^{(2)}$ is simply connected and $(\sphere(V)^{(2)},\F)$ is a full singular Riemannian foliation, by \cite{Lyt10} Corollary 5.3 the quotient $O_c=\sphere(V)^{\reg}/\F=\sigma_c(\sphere(V)^\reg)$ (where $\sphere(V)^\reg$ denotes the union of leaves of maximal dimension in $\sphere(V)^{(2)}$) is a Riemannian orbifold, simply connected
\todo{(21a): ``as an orbifold'' added.}
as an orbifold.
\todo{(21b): part 3: Changed the rest of the proof according to the suggestion of the referee.}
Let $O=\sigma(\sphere(V)^\reg)$. Since different components of a $\p$-fiber have same dimension, it follows that the submetry $\pi:X_c\to X$ restricts to a submetry $\pi:O_c\to O$. Furthermore, for any open set $U\subset O$, the preimage $\pi^{-1}(U)$ equals $\p_c(\p^{-1}(U))$, and thus $\pi|_{\pi^{-1}(U)}:\pi^{-1}(U)\to U$ is a submetry. By Theorem 1.2 of \cite{Lan}, it then follows that $O$ is a Riemannian orbifold as well, and $\pi:O_c\to O$ is a Riemannian orbifold covering. Since $O_c$ is simply connected as an orbifold, it is the universal cover of $O$, and in particular there exists a properly discontinuous, free isometric action of $G=\pi_1^{orb}(O)$ on $O_c$, such that $O_c/G$ is isometric to $O$.

Finally, recall from Lemma \ref{L:small-codim} that $O_c\subset X_c$ and $O\subset X$ are connected and dense, hence every isometry $g:O_c\to O_c$ extends to an isometry $\hat{g}:X_c\to X_c$. In particular, the same group $G$ acts on $X_c$ by isometries, and $X_c/G$ is isometric to $X$.
\end{proof}

%
%

\begin{remark}
\todo[color=cyan]{(22): Added example of action on the quotient that does not lift to an isometric action on the sphere}
In the situation of Theorem \ref{MT:disconnected}, it is not always the case that the $G$-action lifts from $X_c$ to the sphere $\sphere(V)$. For instance, consider $V=\R^6$ as the space of $2\times 3$ matrices, on which the group $\SO(2)\times \SO(3)$ acts by left and right multiplication (see third line of Table E in \cite{GWZ08}), and let $\sigma_c:\sphere(V)\to X_c$ be the corresponding orbit space projection. Then $X_c$ is isometric to an interval of length $\pi/4$, the endpoints of which correspond to the two singular orbits of the $\SO(2)\times \SO(3)$-action. One of the singular isotropy groups is isomorphic to $\SO(2)$, while the other is isomorphic to $\Z_2\times \SO(2)$, which implies that the two singular orbits are not diffeomorphic. Thus the isometric involution of $X_c=[0,\pi/4]$ given by reflection across the midpoint does not lift to an isometry of $\sphere(V)$. Nevertheless, the two singular orbits have the same dimension, so that the composition $\sphere(V)\to X_c\to X_c/\Z_2$ is a (inhomogeneous) manifold submetry.
\end{remark}

Given a manifold submetry $\p:\sphere(V)\to X$ which factors as $\sphere(V)\stackrel{\p_c}{\to} X_c\to X$, by Theorem \ref{MT:disconnected}  we have that $X$ is isometric to $X_c/G$ for some discrete group $G$. We will then say that $\p$ \emph{corresponds to the pair} $(\p_c:\sphere(V)\to X_c, G)$.

\begin{lemma}\label{L:Group-action-on-rings}
Let $\p:\sphere(V)\to X$ a manifold submetry with disconnected fibers, corresponding to the pair $(\p_c:\sphere(V)\to X_c, G)$. Then $G$ induces an action on $A_c=\mathcal{B}(\p_c)$, whose fixed point set is $A=\mathcal{B}(\p)$.
\end{lemma}
\begin{proof}
Let $KX, KX_c$ the Euclidean cones of $X$ and $X_c$ respectively. The manifold submetries $\p, \p_c$ induce manifold submetries $K\p:V\to KX$, $K\p_c:V\to KX_c$. Furthermore, any $g\in G$ induces an isometry $Kg:KX_c\to KX_c$ preserving the codimension of the fibers of $K\sigma_c$.

Define the ring $C^{\infty}(V)^{\p_c}$ of smooth functions which are constant along the $\p_c$-fibers. Since $Kg: KX_c\to KX_c$  preserves the codimension of the fibers of $K\sigma_c$, by Theorem 1.1. of \cite{AR} it induces a map $Kg^*:C^{\infty}(V)^{\p_c}\to C^{\infty}(V)^{\p_c}$ by $Kg^*(f)(p)=f(Kg(\p_c(p)))$, which commutes with the rescalings $r_\lambda:V\to V$, $r_\lambda(v)=\lambda v$. In particular, it takes homogeneous polynomials of degree $d$ in $C^{\infty}(V)^{\p_c}$ to smooth, homogeneous functions $f$ in $C^{\infty}(V)^{\p_c}$ such that $f(r_\lambda(v))=\lambda^df(v)$, i.e., homogeneous polynomials of degree $d$. In other words, $Kg^*$ restricts to a morphism of Laplacian algebras $Kg^*:A_c\to A_c$. Furthermore $K(g_1g_2)^*=Kg_2^*\circ Kg_1^*$ and thus $G$ acts (on the right) on $A_c$. Clearly, $f\in A_c$ is invariant under the $G$-action if and only if $f$ is constant on the unions of $\p_c$-fibers $L_{p_*}=\coprod_{g\in G}\p_c^{-1}(gp_*)$, for any $p_*\in X_c$. However, since $\pi:X_c\to X$ coincides with the quotient by the $G$ action on $X_c$, we have $L_{p_*}=\p^{-1}(\pi(p_*))$ and thus $f\in A_c$ is $G$-invariant if and only if it is constant along the $\p$-fibers.
\end{proof}

The following proposition is a stronger version of Theorem \ref{MT:connected-vs-integral}:
\begin{proposition}
A manifold submetry $\p:\sphere(V)\to X$ has disconnected fibers if and only if $A=\mathcal{B}(\p)$ is not integrally closed in $\RR[V]$. In this case, letting $A_c$ denote the integral closure of $A$ in $\RR[V]$, $\p$ corresponds to the pair $(\p_c:\sphere(V)\to X_c,G)$ where:
\begin{itemize}
\item $\p_c=\mathcal{L}(A_c)$
\item $G$ is the Galois group of the extension of fields of fractions $F(A)\subset F(A_c)$.
\end{itemize}
\end{proposition}

\begin{proof}
\todo{(23): typos fixed.}
Suppose first that $\p:\sphere(V)\to X$ has connected fibers, and let $f\in \RR[V]$ be an integral element over $A=\mathcal{B}(\p)$. Then $f$ is satisfies a polynomial equation
\[
f^n+a_1f^{n-1}+\ldots +a_{n-1}f+a_n=0,\qquad a_1,\ldots, a_n\in A.
\] 
Restricting this equation to a fiber $L$ of $\p$, the restrictions $a_1|_L,\ldots a_n|_L$ are constant, and therefore the restriction $f|_L$ is a solution of a polynomial with constant real coefficients. Since $f$ is continuous and $L$ is connected, it follows that $f$ must be constant on $L$. Since $L$ was chosen arbitrarily, it follows that $f$ is constant along \emph{all} $\p$-fibers, hence $f\in A$ and thus $A$ is integrally closed in $\RR[V]$.

Suppose now that $\p$ has disconnected fibers, with corresponding pair $(\p_c,G)$. Recall that $\p_c:\sphere(V)\to X_c$ is the manifold submetry whose fibers are the connected components of the fibers of $\p$, and $G$ a finite group of isometries of $X_c$ whose quotient is $X$. By the first part of the proof, $A_c=\mathcal{B}(\p_c)$ is integrally closed in $\RR[V]$. We claim that $A\subset A_c$ is an integral extension: in fact, by Lemma \ref{L:Group-action-on-rings}, $G$ acts on $A_c$ with fixed point set $A$. For any $f\in A_c\setminus A$,  define the polynomial in $A_c[t]$:
\[
P(t)=\prod_{g\in G}(t-g\cdot f)=0
\]
This is a monic polynomial, and $f$ satisfies $P(f)=0$. Furthermore, since $g\cdot P=P$, it follows that all the coefficients of $P$ are $G$-invariant, hence $P\in A[t]$. Therefore, $f$ is integral over $A$, hence $A_c$ is the integral closure of $A$ in $\RR[V]$.

It remains to prove that $G$ coincides with the Galois group of the extension $F(A)\subset F(A_c)$, and for this it is enough to show that the field fixed by $G$ is $F(A)$. Let ${a\over b}\in F(A_c)$ an element fixed by $G$, where $a,b\in A_c$. We multiply and divide by $\bar{b}=\prod_{g\in G\setminus \{e\}}g\cdot b$, and obtain
\[
{a\over b}={a\over b}{\bar{b}\over \bar{b}}={a\bar{b}\over \prod_{g\in G}g\cdot b}=:{a'\over b'}
\]
where $b'\in A_c$ is fixed by $G$ and thus $b'\in A$ by Lemma \ref{L:Group-action-on-rings}. But then $a'\in A_c$ is fixed by $G$ as well, and thus again $a'\in A$, which proves ${a\over b}={a'\over b'}\in F(A)$.
\end{proof}

\section{About the maximal and Laplacian conditions}\label{S:maximal-and-laplacian}

Theorem  \ref{MT:1-1correspondence} establishes an equivalence between manifold submetries, and polynomial algebras that are both Laplacian and maximal.

Of these two conditions, being Laplacian is certainly the most compelling one, because it can be fairly easily checked, and it specializes to well-known conditions in two different situations, namely when all generators are quadratic, and when there are exactly two generators. Moreover, in these two situations, Laplacian \emph{implies} maximal, which provides evidence for the Conjecture in the Introduction.
\begin{proposition}
Let $A\subset\R[V]$ be an algebra generated by homogeneous polynomials $\rho_1, \ldots \rho_k$, with $\rho_1=r^2$. Then:
\begin{enumerate}[(a)]
\item $A$ is Laplacian if and only if $\Delta \rho_i, \langle\nabla\rho_i, \nabla\rho_j\rangle\in A$ for every $i,j$.
\item Suppose $\rho_i$ is quadratic for all $i$. Then $A$ is Laplacian if and only if the vector space $\textrm{span}\{\rho_1,\ldots \rho_k\}$ is a Jordan algebra with respect to the product $f\bullet_1 g:=\langle\nabla f, \nabla g\rangle$. In this case, $A$ is maximal.
\item Suppose $k=2$, and let $\tilde{g}=\deg\rho_2$. Then $A$ is Laplacian if and only the generator $\rho_2$ can be replaced with a (homogeneous) polynomial $\tilde{F}$ satisfying the \emph{Cartan-M\"unzner equations} (see \cite[equations (5), (6)]{Muenzner80}):
\[ \Delta \tilde{F}=cr^{\tilde{g}-2}, \qquad \|\nabla\tilde{F}\|^2 = \tilde{g}^2 r^{2\tilde{g}-2}.\]
In this case, $A$ is maximal.
\end{enumerate}
\end{proposition}
\begin{proof}
\begin{enumerate}[(a)]
\item One implication follows immediately from the standard formula for the Laplacian of a product \[\Delta(fg)=f\Delta g + g\Delta f +2\langle\nabla f, \nabla g\rangle.\] 
For the other, assume that $\Delta \rho_i, \langle\nabla\rho_i, \nabla\rho_j\rangle\in A$ for every $i,j$. By linearity it is enough to show that the Laplacian of every monomial $f$ in the $\rho_i$ belongs to $A$. This can be accomplished by proving the following seemingly stronger statement by induction on the length of a monomial $f$ in $\{\rho_i\}$: $\Delta f$ \emph{and}   $\langle\nabla f, \nabla\rho_i\rangle$ belong to $A$, for every $i$.
\item Under the natural identification between quadratic polynomials and self-adjoint endomorphisms, the product $f\bullet_1 g:=\langle\nabla f, \nabla g\rangle$ reduces to the standard Jordan product between self-adjoint endomorphisms $(X,Y)\mapsto (XY+YX)/2$. Since $\Delta \rho_i$ are constant, it follows from part (a) that $A$ is Laplacian if and only if $\textrm{span}\{\rho_1,\ldots \rho_k\}$ is closed under this Jordan product.

\todo{(25): made more precise how to get that $A$ is maximal in this case.}
Theorem B of \cite{MendesRadeschi16}, shows that for any such algebra $A\subseteq \R[V]$, the partition $\F=\mathcal{L}(A)$ is a singular Riemannian foliation, given by the product of Clifford foliations and orbit decompositions of standard diagonal representations. By \cite[Theorem C]{MendesRadeschi16}, all such foliations satisfy the property that $\hat{A}=\mathcal{B}(\F)$ is also generated by degree 2 elements. We then have that $A\subseteq \hat{A}$, and also that $\mathcal{L}(A)=\mathcal{L}(\hat{A})$ which, by Theorem F of \cite{MendesRadeschi16}, implies that $A$ is isomorphic to $\hat{A}$. Therefore, $A=\hat{A}$ and thus $A$ is maximal.

\item Assume first that $A$ is generated by $\rho_1=r^2$ and $\tilde{F}$ satisfying the Cartan-M\"unzner equations. In particular, $\Delta\tilde{F}, \langle\nabla\tilde{F}, \nabla\tilde{F}\rangle\in A$. Since $\Delta r^2$ is a constant, and $ \langle\nabla r^2, \nabla\tilde{F}\rangle = 2\tilde{g} \tilde{F}\in A$, it follows from part (a) that $A$ is Laplacian.

Conversely, suppose $A$ is Laplacian. Then $\Delta \rho_2$ is an element of $A$ that is homogeneous of degree $\tilde{g}-2$, and hence a scalar multiple of $r^{\tilde{g}-2}$.
Similarly, $\|\nabla\rho_2\|^2$ is a linear combination of $r^{\tilde{g}-1}$ and $\rho_2 r^{\tilde{g}-2}$.
If $\tilde{g}$ is odd, it follows that a (non-zero) scalar multiple of $\rho_2$ satisfies the Cartan-M\"unzner equations.
If $\tilde{g}$ is even, we set $\tilde{F}=a \rho_2 + b r^{\tilde{g}}$, and compute $\Delta\tilde{F}$ and $\|\nabla\tilde{F}\|^2$.
It then becomes clear that $a,b\in\R$ can be chosen so that: $\tilde{F}$ satisfies the Cartan-M\"unzner equations; and $a\neq 0$, so that $A$ is also generated by $\rho_1=r^2$ and $\tilde{F}$.

Finally, assume $\tilde{F}$ satisfies the Cartan-M\"unzner equations. To show that the algebra $A$ generated by $r^2$ and $\tilde{F}$ is maximal, first recall that, by \cite[Satz 3]{Muenzner80}, there exists an isoparametric hypersurface $M$ in the sphere $\sphere(V)$, with $g$ principal curvatures, such that the associated so-called Cartan-M\"unzner polynomial $F$ (of degree $g$) satisfies either $\tilde{F}=F$, or $\tilde{F}=\pm(2F^2-r^{2g})$.
The parallel and focal submanifolds to $M$ form an isoparametric foliation $\F$, which is also given by the common level sets $\mathcal{L}(r^2,F)$ of $r^2$ and $F$.

Fix a point $p\in M$ and let $\Sigma\subset V$ be the (two-dimensional) normal space of $M$ at $p$.
Then $\Sigma$ is a \emph{section} of the foliation $\F$, in the sense that every leaf of $\F$ meets $\Sigma$, and does so orthogonally. Clearly, the partition of $\Sigma$ into the intersections of the leaves with $\Sigma$ coincides with $\mathcal{L}(r^2|_\Sigma, F|_\Sigma)$.
Moreover,  $F$ is constructed (see \cite[Section 3]{Muenzner80}) so that $F|_\Sigma(z)=\operatorname{Re} (z^g)$ for all $z\in\C\cong\Sigma$.
It is a well-known fact in Invariant Theory that $|z|^2$ and $\operatorname{Re} (z^g)$ generate the algebra of invariants of the natural action of the dihedral group $D_g$ with $2g$ elements on $\R^2\cong\C$.

Let $h\in\R[V]$ be constant on the common level sets of $A$. We need to show that $h\in A$.

If $\tilde{F}=F$, then $h|_\Sigma$ is $D_g$-invariant, and hence a polynomial in $r^2|_\Sigma$ and $F|_\Sigma$. Since $\Sigma$ meets all leaves of $\F$, this shows that $h\in A$.

If, on the other hand, $\tilde{F}=\pm(2F^2-r^{2g})$, then
\[\tilde{F}|_\Sigma=\pm\left(2\left(\frac{z^g+\bar{z}^g}{2}\right)^2-z^g\bar{z}^g\right)=\pm\operatorname{Re}(z^{2g}).\]
Thus $h|_\Sigma$ is $D_{2g}$-invariant, hence a polynomial in $r^2|_\Sigma$ and $\tilde{F}|_\Sigma$. Since $\Sigma$ meets all common level sets of $\{r^2,\tilde{F}\}$, it follows that $h\in A$.
\end{enumerate}
\end{proof}

%
%
%
%
%
%
%

\begin{appendix}
\section{Lagrangian families of Jacobi fields}\label{A:Lagr}

The goal is this section is to recall some results regarding Lagrangian families of Jacobi fields, and important results by Wilking and Lytchak. For a deeper introduction on this topics, we refer the reader to \cite{Wilking07}, \cite{Lyt09} and Chapter 4 of \cite{RadLN}.

Let $I$ be an interval of any type (it can be a half line or the whole real line as well). Consider a vector bundle $\pi:E\to I$ together with a smoothly-varying inner product  $\scal{\,,\,}$ on each fiber, a covariant derivative $D:\Gamma(E)\to \Gamma(E)$ compatible with the inner product (we will write $X':=D(X)$ for a section $X\in \Gamma(E)$), and a symmetric endomorphism $R\in \Sym^2(E)$ called \emph{curvature operator}. Clearly, given a Riemannian manifold $(M,g)$ and a geodesic $\gamma:I\to M$, then $E=\gamma'^\perp$ automatically comes equipped with $\scal{\,,\,}_t=g_{\gamma(t)}$, $D=\nabla_{\gamma'}$, and $R(t)=R^M(\cdot, \gamma'(t))\gamma'(t)$ where $\nabla$ and $R^M$ denote the Levi Civita connection and the Riemann curvature tensor of $g$, respectively.

Since $I$ is contractible, $E$ is trivial and thus it can be identified, via parallel transport, to $V\times I\to I$ for some Euclidean vector space $(V,\scal{\,,\,})$. Via this identification, $R$ becomes a function $R:I\to \Sym^2(V)$.

With this setup, we can define the space of \emph{($R$-)Jacobi fields} as the set of sections
\[
\mathcal{J}=\{J:I\to V\mid J''(t)+R(t)J(t)=0\quad \forall t\in I\}.
\]
This space has dimension $2\dim V$, isomorphic to $V\oplus V$ via the map $J\mapsto (J(0),J'(0))$. It is easy to see that for $J_1,J_2\in \mathcal{J}$ the function $\omega(J_1,J_2)=\scal{J_1(t),J_2'(t)}-\scal{J_1'(t),J_2(t)}$ is in fact constant, and defines a symplectic product on $\mathcal{J}$.

A subspace $W\subset \mathcal{J}$ is called \emph{isotropic} if $\omega|_W=0$. Equivalently, $W$ is isotropic if $\scal{J_1'(t),J_2(t)}=\scal{J_1(t), J_2'(t)}$  for any $J_1,J_2\in W$. The maximal dimension of an isotropic space is $\dim V$. An isotropic subspace of maximal dimension is called a \emph{Lagrangian subspace}.

Given a subspace $W\subset \mathcal{J}$, define $W_t=\{J\in W\mid J(t)=0\}$ and $W(t)=\{J(t)\mid J\in W\}$. 
One fundamental property of isotropic subspaces is the following:

\begin{proposition}[\cite{Lyt09}, Lemma 2.2]\label{P:equality-ae}
An isotropic space $W$ of Jacobi fields satisfies $\dim W(t)=\dim W$ for all but discretely values of $t$.
\end{proposition}

It follows from the proposition above that the \emph{focal function} $f_W(t):=\dim(W_t)$ equals zero for all but discretely many values of $t\in I$. Thus, it makes sense to define, for every compact interval $[a,b]\subset I$, the \emph{index of $W$ over $[a,b]$} by
\[
\ind_IW=\sum_{t\in [a,b]}f_W(t).
\]

The index satisfies the following semi-continuity property, cf. \cite{Lyt09}:
\begin{proposition}\label{P:continuity}
Let $R_n:I\to \Sym^2(V)$ be a sequence of families of symmetric endomorphisms converging in the $C^0$ topology to $R$. Let $W_n$ be
isotropic subspaces of $R_n$-Jacobi fields that converge to an isotropic subspace $W$ of $R$-Jacobi fields. Let $[a, b]\subseteq I$ be a compact interval and assume that $f_{W_n}(a) = f_{W}(a)$ and $f_{W_n}(b) = f_W (b)$, for all $n$ large enough.

Then $\ind_{[a,b]}W\geq  \ind_{[a,b]}W_n$ for all $n$ large enough. If all $W_n$ are Lagrangians then this inequality becomes an equality.
\end{proposition}

\subsection{Transverse Jacobi equation}\label{A:transvJE}

Let $E\simeq V\times I\to I$ be a vector bundle with $R\in \Sym^2(V)$, and $\Lambda$ be a Lagrangian family of $R$-Jacobi fields, and let $W$ be a subspace of $\Lambda$. Then $W$ is isotropic by default, and by \cite{Wilking07} the subspaces
\[
\tilde{W}(t)=\{J(t)\mid J\in W\}\oplus\{J'(t)\mid J\in W_t\}\subset E_t
\]
define a smooth vector bundle $E_W:=\coprod_{t\in I}\tilde{W}(t)\to I$. The quotient $H:=E/E_W$ comes equipped with:
\begin{itemize}
\item A Euclidean product $\scal{[v_1],[v_2]}:=\scal{pr_{E_W^\perp}(v_1),pr_{E_W^\perp}(v_2)}$, where $pr_{E_W^\perp}E\to E_W^\perp$ denotes the orthogonal projection onto $E_W^\perp$.
\item A covariant derivative $D^H([X(t)])=[D(pr_{E_W^\perp}X(t))]$.
\item A vector bundle map $A: E_W\to H$ given by $A(v)=[J'(t)]$, where $J\in W$ is such that $J(t)=v$.
\item A symmetric endomorphism $R^H\in \Sym^2(H)$ given by $$R^H_t([v])=[R_t(pr_{E_W^\perp}(v))+3AA^*[v]],$$ where $A^*:H\to E_W$ is the adjoint of $A$.
\end{itemize}

\begin{proposition}[Transverse Jacobi equation]
The projection $E\to H$ sends the Jacobi fields in $\Lambda$ to an isotropic subspace of $R^H$-Jacobi fields in $H$, which is isomorphic to $\Lambda/W$ as a vector space.
\end{proposition}

Because of the proposition above, we can identify the quotient $\Lambda/W$ with the corresponding isotropic space of $R^H$-Jacobi fields. Furthermore, by Lemma 3.1 of \cite{Lyt09}, for every $t\in I$ one has
\begin{equation}\label{E:focal-sum}
f_\Lambda(t)=f_W(t)+f_{\Lambda/W}(t)
\end{equation}
and in particular, for every compact subinterval $[a,b]\subset I$,
\begin{equation}\label{E:index-sum}
\ind_{[a,b]}\Lambda=\ind_{[a,b]}W+\ind_{[a,b]}{\Lambda/W}.
\end{equation}

\begin{example}\label{E:example}
Let $\pi:M\to B$ be a Riemannian submersion, $\gamma:I\to M$ a horizontal geodesic, let $\gamma_*=\pi(\gamma)$, and let $E=(\gamma'^\perp)$ be the vector bundle along $I$. Letting $W$ be the (isotropic) space of Jacobi fields along $\gamma$ such that $\pi_*J\equiv0$, it follows by the O'Neill's formulas that $H=E/E_W$ can be canonically identified with $(\gamma_*')^\perp$, in such a way that $R^H(v)=R^B(v, \gamma_*'(t))\gamma_*'(t)$ where $R^B$ denotes the Riemann curvature tensor of $B$.

Furthermore, letting $\Lambda\supseteq W$ denote the (Lagrangian) subspace of Jacobi fields $J$ along $\gamma$, obtained as variation of horizontal geodesics, and such that $\pi_*(J(0))=0$, then $\Lambda/W$ corresponds to the Lagrangian space of Jacobi fields $J_*(t)$ along $\gamma_*$, such that $J_*(0)=0$. In particular, in this case $f_{\Lambda/W}(t)$ counts the conjugate points of $\gamma_*(0)$ along $\gamma_*$.
\end{example}

\section{Manifold submetries}\label{A:manif-subm}

As mentioned in Section \ref{SS:manifold-submetries}, the definitions of singular Riemannian foliation and manifold submetries are very close. The two key features which characterize singular Riemannian foliations are:
\begin{enumerate}
\item The leaves are connected.
\item There is a family of smooth vector fields which span the tangent spaces to the leaves at all points.
\end{enumerate}

A lot of literature has focused mainly on singular Riemannian foliations, and uses the presence of smooth vector fields in several crucial places. The goal of this section is then to re-develop most of the basic results to the case of manifold submetries.
\\

\emph{In this whole section, we will assume $\sigma:M\to X$ is a $C^2$-manifold submetry unless states otherwise.}

\subsection{Homothetic Transformation Lemma, and stratification}\label{SS:HTL}
Let $\p:M\to X$ be a manifold submetry. Since leaves are equidistant, it follows from the first variation formula
\todo{(3): extra word ``first'' deleted}
for the length function that every geodesic starting perpendicular to a leaf, stays perpendicular to all the leaves it meets. Such geodesics are called \emph{horizontal geodesics}.

The first, fundamental result is the following (cf. \cite[Lemma 6.2]{Molino} for transnormal systems):

\begin{lemma}[Homothetic Transformation Lemma]\label{L:HTL}
Let $\p:M\to X$ be a manifold submetry, $L$ a fiber of $\p$, $P\subset L$ a relatively compact open subset of $L$ (called a \emph{plaque}), and let $\epsilon>0$ be small enough that 
for every $v\in \nu^{<\epsilon}P=\{v\in \nu P\mid \|v\|<\epsilon\}$, the geodesic $\gamma_v(t)=\exp(tv)$ minimizes the distance between $\gamma_v(1)$ and $P$. Then for any $\rho_1, \rho_2<\epsilon$ with $\rho_2=\lambda \rho_1$, the map
\[
h_\lambda: \exp(\nu^{\rho_1}P)\to \exp(\nu^{\rho_2}P), \qquad h_\lambda(\exp v):=\exp(\lambda v)
\]
sends fibers of $\p$ into other fibers.
\end{lemma}
\begin{proof}
Let $q=\exp v$, $q'=\exp v' \in \nu^{\rho_1} P$ be points such that $\p(q)=\p(q')=q_*$, and let $\p(P)=p_*$. By construction, the geodesics $\gamma_v(t)=\exp(tv)$ and $\gamma_{v'}(t)=\exp(tv')$ are projected to distance minimizing geodesics from $p_*$ to a bit past $q_*$. Since there is no bifurcation of geodesics in Alexandrov spaces, it follows that $\p(\gamma_v(t))=\p(\gamma_{v'}(t))=:\gamma_*(t)$ and therefore $h_\lambda(q)=\gamma_v(\lambda)$ and $h_\lambda(q')=\gamma_{v'}(\lambda)$ both project to $\gamma_*(\lambda)$.
\end{proof}

For any integer $r$, define $\Sigma^r\subset M$ the union of $\p$-fibers of dimension $r$. Any point $p\in M$ belongs to some stratum $\Sigma^r$, and we define the \emph{stratum through $p$}, and denote it by $\Sigma_p$ the union of connected components of $\Sigma^r$ containing the (possibly disconnected) fiber through $p$. As a direct application of the Homothetic Transformation Lemma, one has

\begin{proposition}[cf. \cite{Molino}, Proposition 6.3]
Given a manifold submetry $M\to X$, for every point $p\in M$ the stratum $\Sigma_p$ is a (possibly non-complete) smooth submanifold of $M$. Furthermore, for any relatively compact open subset $P\subset L$ of the leaf through $p$, there is an $\epsilon$ such that every horizontal geodesic from $p$ initially tangent to $\Sigma_p$ stays in $\Sigma_p$ at least up to distance $\epsilon$.
\end{proposition}

\begin{remark}\label{R:different-components-same-dimension}
It is important to notice that, in particular, if $\Sigma_p$ is disconnected, then different components will still have the same dimension.
\todo{(4): full stop added}
\end{remark}

\begin{lemma}\label{L:transverse-int}
Let $\p:M\to X$, $L$, $P$, and $\epsilon$ as above. Consider the closest-point map $f:\exp \nu^{<\epsilon}P\to P$. Given a $\p$-fiber $L'$ intersecting $\exp \nu^\epsilon P$, let $P':=L'\cap \exp\nu^{<\epsilon}P$ and $f'$ be the restriction of $f$ to $P'$. Then:
\todo{(5): restated lemma to define $L'$.}
\begin{enumerate}
\item The differential $d_qf'$ is surjective.
\item For any $p\in P$ and $x\in \nu^{<\epsilon}_pP$, the fiber $L'$ through $q:=\exp x$ is transverse to the slice $D_p:=\exp \nu^{<\epsilon}_pP$ at $q$.
\item The function $M\to \mathbb{Z}$, $p\mapsto \dim(L_p)$, is lower semicontinuous.
\end{enumerate}
\end{lemma}
\begin{proof}
1) Let $\gamma(t)=\exp tx$. For any vector $v\in T_qP'$, let $J_v(t)$ the Jacobi field defined by $J_v(t)=(h_{t})_*v$. By the Homothetic Transformation Lemma, $J_v(t)$ is tangent to the $\p$-fibers for all $t\in [0,1]$. In particular, $J_v(0)\in T_pP$. Let $W=\{J_v\mid v\in T_qP'\}$. Notice that $W$ is contained in the Lagrangian family $\Lambda_L$ consisting of Jacobi fields generated by variations of normal geodesics through $L$ (cf. Appendix \ref{A:Lagr}). In particular, $W$ is isotropic and any  $J_v\in W$ vanishing at $0$ satisfies $J_v'(0)\perp T_pL$.
\\

We can also embed $W$ in the Lagrangian space $\Lambda_{L'}$  of Jacobi fields generated by variations of horizontal geodesics through $L'$. Letting $\Lambda_0=\{J\in \Lambda_{L'}\mid J(1)=0,\,J'(1)\perp T_qP'\}$, we have $\Lambda_{L'}=\Lambda_0\oplus W$. By Section \ref{A:transvJE} we get
\[
\gamma'(t)^{\perp}=\{J(t)\mid J\in\Lambda_{L'}\}\oplus\{J'(t)\mid J\in \Lambda_{L'}, J(t)=0\}
\]
In particular, every $w\in T_pP$ can be written as
\begin{equation}\label{E:vector-decomposition}
w=J_{u}(0)+J_{v}'(0)+J_3(0)+J_4'(0),
\end{equation}
where $J_{u}, J_v\in W$, $J_3,J_4\in \Lambda_0$, and $J_v(0)=J_4(0)=0$. Notice that:
\begin{itemize}
\item $J_4=0$ because otherwise $p$ and $q$ would be conjugate points.
\item By the discussion above, $J_u(0)\in T_pP$ and $J_v'(0)\in \nu_pP$.
\item Taking the projection of Equation \ref{E:vector-decomposition} onto $\nu_pP$ and using the previous points, we get
\[
0=J_v'(0)+pr_{\nu_pP}J_3(0)
\]
However, by the definition of Lagrangian space of Jacobi fields,
\[
-\|J_v'(0)\|^2=\scal{J_v'(0), pr_{\nu_pP}J_3(0)}=\scal{J_v'(0), J_3(0)}=\scal{J_v(0), J_3'(0)}=0
\]
and thus $J_v'(0)=0$ and $J_3(0)\in T_pP$.
\item $J_3=0$, because otherwise $q$ would be a focal point for $q$, which is not possible because $\gamma$ keeps minimizing past $q$.
\end{itemize}
Therefore, it must be $w=J_u(0)$. Notice however that $J_u(0)=d_qf'(u)$ and therefore $d_qf':T_qP'\to T_pP$ is surjective.
\\

2)  Since the kernel of $d_qf:T_qM\to T_pL$ is $T_qD_p$ and $d_qf$ is surjective by the previous point, the result follows.

3) It is enough to prove that for every $p\in M$ there is a neighborhood $U$ around $p$ such that $\dim L_q\geq \dim L_p$ for every $q\in U$. This is exactly what point 1) shows.
\end{proof}

\begin{remark}
\begin{enumerate}
\item In the case of singular Riemannian foliations, the semicontinuity of the dimension of leaves follows immediately from the existence of smooth vector fields spanning the foliation.
\item Lemma \ref{L:transverse-int} shows that for every $r_0$, the union $\bigcup_{r\geq r_0}\Sigma^{r}$ is open. In particular, the \emph{regular part}, consisting of fibers of maximal dimension, is open in $M$.
\end{enumerate}
\end{remark}


\subsection{Generic strata}

In this section we assume that $\p:M\to X$ is a smooth manifold submetry with connected fibers, and let $M^{(2)}$ be the union of the strata $\Sigma_p$ of codimension $\leq 2$ (see Section \ref{SS:HTL}). The main result of this section will be to show that the fibers of $\p$ form a \emph{full} singular Riemannian foliation on $M^{(2)}$ (see Definition \ref{D:full} below).

\begin{lemma}\label{L:small-codim}
There are no strata of codimension 1. Moreover, let $\p:M\to X$ a manifold submetry, and $\Sigma_p$ be a stratum of codimension 2. Let $U$ be a relatively compact neighborhood of $p$ in $\Sigma_p$, and $\epsilon$ small enough that all normal geodesics from $U$ minimize the distance from $\Sigma_p$ up to time $\epsilon$. Let $B_\epsilon(U)=\exp \nu^{<\epsilon}(U)$. Then for any $q=\exp_{p'}v\in B_\epsilon(U)\setminus U$, $v\in \nu_{p'}^{<\epsilon}(U)$, the $\p$-fiber through $q$ is given by $S_d(L_{p'})\cap S_d(U)$ where $d=dist(q,U)$ and $S_d(L_{p'})$ (resp. $S_d(U)$) denotes the boundary of the tube of distance $d$ around $L_{p'}$ (resp. around $U$).
\end{lemma}
\begin{proof}
First of all notice that $q\notin \Sigma_p$ and, by Lemma \ref{L:transverse-int} and the Homothetic Transformation Lemma, $\dim(L_q)>\dim L_{p'}$. By definition of $\epsilon$, it follows $d=\dist(q,U)=\dist(q, L_{p'})=\dist(q,p')$.
\todo{(6): lines added to compute the dimension of $S_d(L_{p'})\cap S_d(U)$.}
Notice furthermore that $S_d(L_{p'})\cap S_d(U)=\exp \nu^dU\big|_{L_{p'}}$ is a manifold and, since $U$ has codimension 2 in $M$ by assumption, one has
\[
\dim S_d(L_{p'})\cap S_d(U)=\dim \exp \nu^dU\big|_{L_{p'}}= \dim \nu^dU\big|_{L_{p'}}= \dim L_{p'}+1.
\]
By equidistance of the $\p$ fibers, the fiber $L_q$ through $q$ must lie in $S_d(L_{p'})$. Moreover since $\dist(\cdot,U)=\inf_{r\in U}\dist(\cdot,L_{r})$, it follows that the distance from $U$ is constant along the $\p$-fibers in $B_\epsilon(U)$, and thus $L_q$ must lie in $S_d(U)$ as well. Therefore, $L_q\cap B_\epsilon(U)$ is contained in $S_d(L_{p'})\cap S_d(U)$. On the other hand, it is easy to see that $\dim(S_d(L_{p'})\cap S_d(U))=\dim L_p+\operatorname{codim} \Sigma_p-1$. Thus, if $\Sigma_p$ was a stratum of codimension 1, then $\dim L_q\leq \dim L_p$ which would give a contradiction, hence there are no strata of codimension 1. Now letting $\Sigma_p$ be a stratum of codimension $2$, this would imply that $\dim (L_q\cap B_{\epsilon}(U))\leq \dim(S_d(L_{p'})\cap S_d(U))=\dim L_{p'}+1$ and the only possibility is that the inequality is in fact an equality, in which case $L_q\cap B_{\epsilon}(U)=S_d(L_{p'})\cap S_d(U)$.
\end{proof}

Recall Definition \ref{D:srf} for the notion of singular Riemannian foliation. We now define the concept of \emph{full singular Riemannian foliation} (see \cite{Lyt10}):
\begin{definition}\label{D:full}
A singular Riemannian foliation $(M,\F)$ is called \emph{full} if for every point $p\in M$ there exists an $\epsilon$, such that the normal exponential map $\exp:\nu^{<\epsilon}L\to M$ from the leaf $L$ through $p$ is well defined.
\end{definition}
We can now prove the following:
\begin{proposition}\label{P:orbifold-part-srf}
Let $\p:M\to X$ be a manifold submetry with connected fibers. Then the partition $(M^{(2)},\F)$ of $M^{(2)}$ into the fibers of $\p|_{M^{(2)}}$ is a full singular Riemannian foliation.
\end{proposition}
\begin{proof}
Since the fibers of $\p$ are connected by assumption, the only thing to prove is that every vector $v$ tangent to a $\p$-fiber, can be locally extended to a vector field everywhere tangent to the $\p$-fibers. Once this is proved, the foliation is automatically full since every leaf $L$ of $\F$ is compact.

Fix $p\in M^{(2)}$, let $L$ be the $\p$-fiber through $p$ and $\Sigma_p$ the stratum through $p$, and fix $v\in T_pL$. Clearly if the codimension of $\Sigma_p$ is zero, then $\p$ is a Riemannian submersion around $L$ and it is straightforward to produce a local vector field $V$ everywhere tangent to the $\p$-fibers extending $v$. Furthermore, by Lemma \ref{L:small-codim} $\Sigma_p$ cannot have codimension 1, which only leaves the case of $\Sigma_p$ having codimension 2. In this case, $\p|_{\Sigma_p}$ is still a Riemannian submersion, and any vector $v\in T_pL$ can be extended to a vector field $V_1$ in $\Sigma_p$, tangent to the $\p$-fibers. Take a neighborhood $U$ of $p$ in $\Sigma_p$, and let $\epsilon$ small enough, as in Lemma \ref{L:small-codim}. We can extend $V_1$ to a vector field $V$ in $B_\epsilon(U)$ as follows: first take any extension $V_2$ of $V_1$ to $B_{\epsilon}(U)$. Secondly, define the \emph{linearization of $V_2$ along $U$} as
\[
V_2^{\ell}=\lim_{\lambda\to 0}(h_{\lambda})_*^{-1}(V_2\circ h_{\lambda}),
\]
where $h_\lambda:B_\epsilon(U)\to B_{\lambda \epsilon}(U)$ denotes the homothetic transformation $\exp v\mapsto \exp \lambda v$, $v\in \nu^{<\epsilon}U$. By the properties of linearized vector fields
\todo{(7a): made precise reference to the proposition we need here}
(cf. \cite{MR}, Proposition 13) $V_2^\ell$ is still smooth and it projects to $V_1$ via the closest-point-map projection $\pi:B_\epsilon(U)\to U$. In particular, letting $K\subset T B_\epsilon(U)$ be the smooth distribution given by $\ker(\pi_*)$, the projection $V=pr_{K^\perp}V_2^\ell$ is
\todo{(7b): changed the sentence, hoping that it is less confusing}
the unique vector field perpendicular to the $\pi$-fibers which projects to $V_1$ via $\pi$ (that is, $V$ the \emph{horizontal  extension} of $V_1$ with respect to the submersion  $\pi:B_{\epsilon}(U)\to U$). It then follows that the flows $\Phi^t_V$, $\Phi^t_{V_1}$ satisfy $\pi\circ \Phi_{V}^t=\Phi_{V_1}^t\circ \pi$.
\todo{(7c): changed the sentence, trying to make it less confusing.}
The flow lines of $V$ (which stay at a constant distance from $U$ by the first variation of length) are thus also equidistant to $L'$ thus $V$ is tangent to the intersections $S_d(U)\cap S_d(L')$. These, by Lemma \ref{L:small-codim}, coincide with the leaves $L_q\cap B_{\epsilon}(U)$, $q\in B_\epsilon(U)$.


Summing up $V$ is a local vector field, everywhere tangent to the leaves, which coincides with the vector $v$ at $p$. Since $v$ was arbitrary, $(M^{(2)},\F)$ is a (full) singular Riemannian foliation.
\end{proof}

\end{appendix}

\bibliographystyle{alpha}
\bibliography{ref}
\end{document}